\documentclass[11pt]{amsart}
\usepackage[UKenglish]{isodate}
\usepackage{amsmath}
\usepackage{amsthm}
\usepackage{amsbsy}
\usepackage{amssymb}
\usepackage{amsfonts}
\usepackage[dvips]{lscape}
\usepackage{xcolor}
\usepackage{amscd}
\usepackage[all,cmtip]{xy}
\usepackage{euscript}
\usepackage{parskip}
\usepackage{enumerate}
\usepackage{yfonts}
\usepackage{xcolor}
\usepackage{graphicx}
\usepackage{fancyhdr}
\usepackage{tikz-cd}
\usepackage{slashbox}
\usepackage{colortbl}

\usepackage{fullpage}
\usepackage{verbatim}

\usetikzlibrary{positioning}
\usetikzlibrary{decorations.pathreplacing}

\usepackage[margin=1in]{geometry}
\usepackage[expansion=false]{microtype}
\usepackage[pdfusetitle,unicode,hidelinks]{hyperref}

\theoremstyle{plain}

\theoremstyle{plain}
\newtheorem{theorem}{Theorem}[section]

\newtheorem{proposition}[theorem]{Proposition}
\newtheorem{lemma}[theorem]{Lemma}
\newtheorem{corollary}[theorem]{Corollary}

\theoremstyle{remark}
\newtheorem{remark}[equation]{Remark}

\theoremstyle{definition}

\newcommand{\cal}{\EuScript}

\renewcommand{\vec}[1]{\boldsymbol{#1}}
\newcommand{\vol}{\textup{vol}}

\newif\iffinalrun
\iffinalrun
\else
 \fi

\iffinalrun
  \newcommand{\need}[1]{}
  \newcommand{\mar}[1]{}
\else
  \newcommand{\need}[1]{{\tiny *** #1}}
  \newcommand{\mar}[1]{\marginpar{\raggedright\tiny Fix Me:  #1 }}\fi

\usepackage{amssymb,amsmath,amsfonts,amsthm,epsfig,amscd,stmaryrd}
\usepackage{stmaryrd}

\renewcommand{\contentsname}{}
\begin{document}
\title{New upper bounds for spherical codes and packings}
\author{Naser Talebizadeh Sardari and Masoud Zargar}
\address{Penn State department of Mathematics, McAllister Building, Pollock Rd, State College, PA 16802 USA}
\email{nzt5208@psu.edu}

\address{Max Planck Institute for Mathematics, Vivatsgasse 7, 53111 Bonn, Germany}
\email{mzargar@usc.edu}
\renewcommand{\contentsname}{}
\begin{abstract}
We  improve the previously  best  known upper bounds on the sizes of $\theta$-spherical codes for every $\theta<\theta^*\approx 62.997^{\circ}$
 at least by a factor of $0.4325$, in sufficiently high dimensions. Furthermore, for sphere packing densities in dimensions $n\geq 2000$ we have an improvement at least by a factor of $0.4325+\frac{51}{n}$. Our method also breaks many non-numerical sphere packing density bounds in smaller dimensions. This is the first such improvement for each dimension since the work of Kabatyanskii and Levenshtein~\cite{KL} and its later improvement by Levenshtein~\cite{Leven79}. Novelties of this paper include the analysis of triple correlations, usage of the concentration of mass in high dimensions, and the study of the spacings between the roots of Jacobi polynomials.
\end{abstract}
\maketitle
\setcounter{tocdepth}{1}
\tableofcontents
\section{Introduction}\label{intro}
\subsection{Spherical codes and packings}
Packing densities have been studied extensively, for purely mathematical reasons as well as for their connections to coding theory. The work of Conway and Sloane is a comprehensive reference for this subject~\cite{Conway}. We proceed by defining the basics of this subject.  
Consider $\mathbb{R}^n$ equipped with the Euclidean metric $|.|$ and the associated volume $\vol(.)$. For each real $r>0$ and each $x\in\mathbb{R}^n$, we denote by $B_n(x,r)$ the open ball in $\mathbb{R}^n$ centered at $x$ and of radius $r$. For each discrete set of points $S\subset\mathbb{R}^n$ such that any two distinct points $x,y\in S$ satisfy $|x-y|\geq 2$, we can consider
\[\cal{P}:=\cup_{x\in S}B_n(x,1),\]
the union of non-overlapping unit open balls centered at the points of $S$. This is called a \textit{sphere packing} ($S$ may vary), and we may associate to it the function mapping each real $r>0$ to
\[\delta_{\cal{P}}(r):=\frac{\vol(\cal{P}\cap B_n(0,r))}{\vol(B_n(0,r))}.\]
The \textit{packing density} of $\cal{P}$ is defined as
\[\delta_{\cal{P}}:=\limsup_{r\rightarrow\infty}\delta_{\cal{P}}(r).\]
Clearly, this is a finite number. The \textit{maximal sphere packing density} in $\mathbb{R}^n$ is defined as
\[\delta_n:=\sup_{\cal{P}\subset\mathbb{R}^n}\delta_{\cal{P}},\]
a supremum over all sphere packings $\cal{P}$ of $\mathbb{R}^n$ by non-overlapping unit balls.\\
\\
The linear programming method initiated by Delsarte is a powerful tool for giving upper bounds on sphere packing densities~\cite{Delsarte1}.
That being said, we only know the optimal sphere packing densities in dimensions 1,2,3,8 and 24 \cite{Fejes,Hales,Maryna1,Maryna2}. Very recently, the first author proved an optimal upper bound on the sphere packing density of all but a tiny fraction of even unimodular lattices in high dimensions; see~\cite[Theorem 1.1]{Sardari}.\\
\\
The best known \textit{linear} programming upper bounds on sphere packing densities in low dimensions are based on the linear programming method developed by Cohn--Elkies~\cite{Elkies} which itself was inspired by Delsarte's linear programming method. As far as the exponent is concerned, in high dimensions, the best asymptotic upper bound goes back to Kabatyanskii--Levenshtein from 1978~\cite{KL} stating that $\delta_n\leq 2^{-(0.599+o(1))n}$ as $n\rightarrow\infty$. More recently, de Laat--de Oliveira Filho--Vallentin  improved upper bounds in very low dimensions using the semi-definite programming method~\cite{Frank}, partially based on the semi-definite programming method developed by Bachoc--Vallentin~\cite{BV} for bounding kissing numbers. The work of Bachoc--Vallentin was further improved by Mittelmann--Vallentin~\cite{MittelVallentin}, Machado--de Oliveira Filho~\cite{MachadoOliveira}, and very recently after the writing of our paper by de Laat--Leijenhorst~\cite{LL}.\\
\\
Another recent development is the discovery by Hartman--Maz\'ac--Rastelli~\cite{HMR} of a connection between the spinless modular bootstrap for two-dimensional conformal field theories and the linear programming bound for sphere packing densities. After the writing of our paper, Afkhami-Jeddi--Cohn--Hartman--de Laat--Tajdini~\cite{ACHLT} numerically constructed solutions to the Cohn--Elkies linear programming problem and conjectured that the linear programming method is capable of producing an upper bound on sphere packing densities in high dimensions that is exponentially better than that of Kabatyanskii--Levenshtein.\\
\\
A notion closely related to sphere packings in Euclidean spaces is that of spherical codes. By inequalities relating sphere packing densities to the sizes of spherical codes, Kabatyanskii--Levenshtein~\cite{KL} obtained their bound on sphere packing densities stated above. The sizes of spherical codes are bounded from above using Delsarte's linear programming method. In what follows, we define spherical codes and this linear programming method.\\
\\
Given $S^{n-1}$, the unit sphere in $\mathbb{R}^n$, a $\theta$-\textit{spherical code} is a finite subset $A\subset S^{n-1}$  such that no two distinct $x,y\in A$ are at an angular distance less than $\theta$. For each $0<\theta\leq\pi$, we define $M(n,\theta)$ to be the largest cardinality of a $\theta$-spherical code $A\subset S^{n-1}.$\\
\\
The Delsarte linear programming method is applied to spherical codes as follows. Throughout this paper, we work with \textit{probability} measures $\mu$ on $[-1,1].$ $\mu$  gives an inner product on the $\mathbb{R}$-vector space of real polynomials $\mathbb{R}[t]$, and let $\{p_i\}_{i=0}^{\infty}$ be an orthonormal basis with respect to $\mu$ such that the degree of $p_i$ is $i$ and $p_i(1)>0$ for every $i$. Note that $p_0=1$. Suppose that the basis elements $p_k$ define positive definite functions on $S^{n-1}$, that is,
\begin{equation}\label{positivedef}\sum_{x_i,x_j\in A}h_ih_jp_k\left(\left<x_i,x_j\right>\right)\geq 0
\end{equation}
for any finite subset $A\subset S^{n-1}$ and any real numbers $h_i\in \mathbb{R}.$ An example of a probability measure satisfying inequality~\eqref{positivedef} is
\[d\mu_{\alpha}=\frac{(1-t^2)^{\alpha}}{\int_{-1}^1(1-t^2)^{\alpha}dt} dt,\]
where $\alpha\geq\frac{n-3}{2}$ and $2\alpha\in \mathbb{Z}$. Given $s\in[-1,1]$, consider the space $D(\mu,s)$ of all functions $f(t)=\sum_{i=0}^{\infty}f_ip_i(t)$, $f_i\in\mathbb{R}$, such that
\begin{enumerate}\label{Del}
\item $f_i\geq 0$ for every $i$, and $f_0>0,$ \label{newlp}
\item $f(t)\leq 0$ for $-1\leq t\leq s.$
\end{enumerate}
Suppose $0<\theta<\pi$, and $A=\{x_1,\hdots,x_N\}$ is a $\theta$-spherical code in $S^{n-1}$. Given a function $f\in D(\mu,\cos\theta)$, we  consider
\[\sum_{i,j}f(\left<x_i,x_j\right>).\]
This may be written in two different ways as
\[Nf(1)+\sum_{i\neq j}f(\left<x_i,x_j\right>)=f_0N^2+\sum_{k=1}^{\infty}f_k\sum_{i,j}p_k(\left<x_i,x_j\right>).\]
Since $f\in D(\mu,\cos\theta)$ and $\left<x_i,x_j\right>\leq \cos\theta$ for every $i\neq j$, this gives us the inequality
\[N\leq\frac{f(1)}{f_0}.\]
We define 
\begin{equation}\label{Ldef}\cal{L}(f):=\frac{f(1)}{f_0}.\end{equation}
In particular, this method leads to the linear programming bound
\begin{equation}\label{Dbound}M(n,\theta)\leq \inf_{f\in D(d\mu_{\frac{n-3}{2}},\cos\theta)}\cal{L}(f).
\end{equation}
One of the novelties of our work is the construction using triple points of new test functions in Section~\ref{newtest} satisfying conditions (1) and (2) of the Delsarte linear programming method. In fact, our functions are infinite linear combinations of coefficients of the matrices appearing in Theorem 3.2 of Bachoc--Vallentin~\cite{BV}. Bachoc--Vallentin use semi-definite programming to obtain an upper bound on kissing numbers $M(n,\frac{\pi}{3})$ by summing over triples of points in spherical codes. On the other hand, we average one of the three points over the sphere, and take the other two points from the spherical code. Semi-definite programming is computationally feasible in very low dimensions, and improves upon linear programming bounds~\cite{Viazlow}. After the writing of our paper, a semi-definite programming method using triple point correlations for sphere packings was developed by Cohn--de Laat--Salmon~\cite{CDS}, improving upon linear programming bounds on sphere packings in special low dimensions. In the semi-definite programming methods, the functions are numerically constructed. Furthermore, in high dimensions, there is no asymptotic bound using semi-definite programming which improves upon the linear programming bound of Kabatyanskii--Levenshtein~\cite{KL}, even up to a constant factor. The functions that we non-numerically construct improve upon~\cite{KL} by a constant factor. In the same spirit, we also improve upon sphere packing density upper bounds in high dimensions.\\
\\
Upper bounds on spherical codes are used to obtain upper bounds on sphere packing densities through inequalities proved using geometric methods. For example, for any $0<\theta\leq\pi/2$, Sidelnikov~\cite{nikov} proved using an elementary argument that
\begin{equation}\label{side}\delta_n\leq \sin^n(\theta/2)M(n+1,\theta).\end{equation}
Let $ 0< \theta< \theta^{\prime} \leq \pi$. We write $\lambda_n(\theta,\theta^{\prime})$ for the ratio of volume of the spherical cap with radius $\frac{\sin(\theta/2)}{\sin(\theta^{\prime}/2)}$ on the unit sphere $S^{n-1}$ to the volume of the whole sphere. Sidelnikov~\cite{nikov} used a similar argument to show that for $0<\theta<\theta'\leq\pi$
\begin{equation}\label{sid}M(n,\theta)\leq \frac{M(n+1,\theta')}{\lambda_n(\theta,\theta^{\prime})}.\end{equation}
Kabatyanskii--Levenshstein used the Delsarte linear programming method and Jacobi polynomials to give an upper bound on $M(n,\theta)$~\cite{KL}. A year later, Levenshtein~\cite{Leven79} found optimal polynomials up to a certain degree and improved the Kabatyanskii--Levenshtein bound by a constant factor. Levenshtein obtained the  upper bound  
\begin{equation}
\label{levinq}
M(n,\theta)\leq M_{\textup{Lev}}(n,\theta),
\end{equation} where 
\begin{equation}\label{defmlev}
M_{\textup{Lev}}(n,\theta):=\begin{cases}
2 {\binom{d+n-1}{n-1}} &\text{ if } t^{\alpha+1,\alpha}_{1,d}< \cos(\theta) \leq t^{\alpha+1,\alpha+1}_{1,d} ,
\\
\binom{d+n-1}{n-1}+\binom{d+n-2}{n-1}  &\text{ if } t^{\alpha+1,\alpha+1}_{1,d-1}< \cos(\theta) \leq t^{\alpha+1,\alpha}_{1,d}.
\end{cases}
\end{equation}
Here, $\alpha=\frac{n-3}{2}$ and $t^{\alpha,\beta}_{1,d}$ is the largest root of the Jacobi polynomial $p^{\alpha,\beta}_d$ of degree $d=d(n,\theta)$, a function of $n$ and $\theta$. We carefully define $d=d(n,\theta)$ and Levenshtein's optimal polynomials in Subsection~\ref{local}. Also see Subsection~\ref{local} for the definition and properties of Jacobi polynomials.\\
\\
Throughout this paper, $\theta^*= 62.997...^{\circ}$ is the unique root of the equation 
\[\cos \theta \log(\frac{1+\sin\theta}{1-\sin\theta})-(1+\cos\theta)\sin\theta=0\]
in $(0,\pi/2)$. As demonstrated by Kabatyanskii--Levenshtein in ~\cite{KL}, for $0<\theta< \theta^*$, the bound $M(n,\theta)\leq M_{\textup{Lev}}(n,\theta)$ is asymptotically exponentially weaker in $n$ than
\begin{equation}\label{sidimproved}M(n,\theta)\leq \frac{M_{\textup{Lev}}(n+1,\theta^*)}{\lambda_n(\theta,\theta^*)}\end{equation}
obtained from inequality~\eqref{sid}. Barg--Musin~\cite[p.11 (8)]{BargMusin}, based on the work \cite{Agrell} of Agrell--Vargy--Zeger, improved inequality~\eqref{sid} and showed that
\begin{equation}\label{Agrell}
M(n,\theta)\leq \frac{M(n-1,\theta')}{\lambda_n(\theta,\theta')}
\end{equation}
whenever $\pi>\theta'>2\arcsin\left(\frac{1}{2\cos(\theta/2)}\right)$. Cohn--Zhao \cite{CohnZhao} improved sphere packing density upper bounds by combining the upper bound of Kabatyanskii--Levenshtein on $M(n,\theta)$ with their analogue \cite[Proposition 2.1]{CohnZhao} of~\eqref{Agrell} stating that for $\pi/3\leq\theta\leq \pi$,
\begin{equation}\label{CohnZhaoineq}\delta_n\leq \sin^n(\theta/2)M(n,\theta),
\end{equation}
leading to better bounds than those obtained from \eqref{side}. Aside from our main results discussed in the next subsection, in Proposition~\ref{gmim} of Section~\ref{GM} we remove the angular restrictions on inequality~\eqref{Agrell} in the large $n$ regime by using a concentration of mass phenomenon. An analogous result removes the restriction $\theta\geq \pi/3$ on inequality~\eqref{CohnZhaoineq} for large $n$.\\
\\
Inequalities~\eqref{Agrell} and~\eqref{CohnZhaoineq} give, respectively, the bounds
\begin{equation}\label{Agrell2}
M(n,\theta)\leq \frac{M_{\textup{Lev}}(n-1,\theta^*)}{\lambda_n(\theta,\theta^*)},
\end{equation}
for $\theta$ and $\theta^*$ restricted as in the conditions for inequality~\eqref{Agrell}, and
\begin{equation}\label{CohnZhaoineq2}\delta_n\leq \sin^n(\theta^*/2)M_{\textup{Lev}}(n,\theta^*).
\end{equation}
Prior to our work, the above were the best bounds for large enough $n$. In Theorems~\ref{mainbound} and~\ref{spherepacking}, we improve both by a constant factor for large $n$, the first such improvement since the work of Levenshtein~\cite{Leven79} more than forty years ago. We also relax the angular condition in inequality~\eqref{Agrell2} to $0<\theta<\theta^*$ for large $n$. For $\theta^*\leq\theta\leq\pi$, $M(n,\theta)\leq M_{\textup{Lev}}(n,\theta)$ is still the best bound. We also prove a number of other results, including the construction of general test functions in Section~\ref{newtest} that are of independent interest.
\subsection{Main results and general strategy}
We improve inequalities~\eqref{Agrell2} and~\eqref{CohnZhaoineq2}   with an extra factor $0.4325$ for each sufficiently large $n$. In the case of sphere packings,  we obtain an improvement by a factor of $0.4325+\frac{51}{n}$ for dimensions $n\geq 2000$. In low dimensions, our geometric ideas combined with numerics lead to improvements that are better than $0.4325$. In Section~\ref{numerics}, we provide the results of our extensive numerical calculations. We now state our main theorems. 
\begin{theorem}\label{mainbound} Suppose that 
 $0<\theta<\theta^*$. Then 
\[
M(n,\theta)\leq c_n\frac{M_{\textup{Lev}}(n-1,\theta^*)}{\lambda_n(\theta,\theta^*)},
\]
where $c_n\leq 0.4325$ for large enough $n$ independent of $\theta.$
\end{theorem}

We also have a uniform version of this theorem  for sphere packing densities.
\begin{theorem}\label{spherepacking}Suppose that $\frac{1}{3}\leq \cos(\theta)\leq \frac{1}{2}$. We have 
\[\delta_{n} \leq  c_{n}(\theta)\sin^{n}(\theta/2)M_{\textup{Lev}}(n,\theta),
\]
where $c_{n}(\theta)<1$ for every $n$.
 If, additionally, $n\geq 2000$ we have $c_{n}(\theta)\leq 0.515+\frac{74}{n}$. 
 Furthermore, for $n\geq 2000$ we have $c_{n}(\theta^*)\leq 0.4325+\frac{51}{n}$.
\end{theorem}
By Kabatyanskii--Levenshtein~\cite{KL}, the best bound on sphere packing densities $\delta_n$ for large $n$ comes from $\theta=\theta^*$; comparisons using other angles are exponentially worse. Consequently, this theorem implies that we have an improvement by $0.4325$ for sphere packing density upper bounds in high dimensions. Furthermore, note that the constants of improvement $c_n(\theta)$ are bounded from above \textit{uniformly} in $\theta$. The lower bound in $\frac{1}{3}\leq \cos\theta\leq\frac{1}{2}$ is not conceptually significant in the sense that a change in $\frac{1}{3}$ would lead to a change in the bound $c_n(\theta)\leq 0.515+\frac{74}{n}$.\\
\\
We prove Theorem~\ref{spherepacking} by constructing a new test function that satisfies the Cohn--Elkies linear programming conditions. 
\begin{figure}[h!]
	\centering
		\includegraphics[width=50mm,scale=0.5]{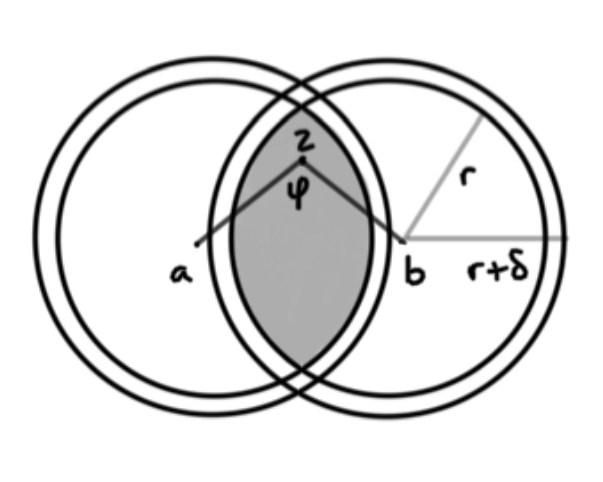}
	\label{fig:image1}
	\caption{Schematic diagram for sphere packings.}\label{img1}
\end{figure}
The geometric idea behind the construction of our new test functions is the following (see Figure~\ref{img1}). In~\cite{CohnZhao}, for every $\pi\geq\theta\geq\frac{\pi}{3}$, Cohn and Zhao choose a ball of radius $r=\frac{1}{2\sin(\theta/2)}$ around each point of the sphere packing so that for every two points $\boldsymbol{a}$ and $\boldsymbol{b}$ that are centers of balls in the sphere packing, for every point $\boldsymbol{z}$ in the shaded region, $\boldsymbol{a}$ and $\boldsymbol{b}$ make an angle $\varphi\geq\theta$ with respect to $\boldsymbol{z}$. By averaging a function satisfying the Delsarte linear programming conditions for the angle $\theta$, a new function satisfying the conditions of the Cohn--Elkies linear programming method is produced. The negativity condition easily follows as $\varphi\geq\theta$ for every point $\boldsymbol{z}$ in the shaded region. Our insight is that since we are taking an average, we do not need pointwise negativity for each point $\boldsymbol{z}$ to ensure the negativity condition for the new averaged function. In fact, we can enlarge the radii of the balls by a quantity $\delta=O(1/n)$ so that the conditions of the Cohn--Elkies linear programming method continue to hold for this averaged function. Since the condition $\varphi\geq \theta$ is no longer satisfied for every point $\boldsymbol{z}$, determining how large $\delta$ may be chosen is delicate. We develop analytic methods for determining such $\delta$. This requires us to estimate triple density functions in Section~\ref{tripleden}, and
estimating the Jacobi polynomials near their extreme roots in Subsection~\ref{local}. It is known that the latter problem is difficult~\cite[Conjecture 1]{Krasikov}.  In Subsection~\ref{local}, we treat  this difficulty  by using the relation between the zeros of  Jacobi polynomials; these ideas go back to the work of Stieltjes~\cite{Stieltjes}. More precisely, we use the underlying differential equations satisfied by Jacobi polynomials, and the fact that the roots of the family of Jacobi polynomials are interlacing. For Theorem~\ref{mainbound}, we use a similar idea but consider how much larger we can make certain appropriately chosen strips on a sphere. See Figure~\ref{img2}.\\
\\

\begin{table}[t!]
\begin{tabular}{c|l|l|l|l|l|l}
\hline
$n$ & \textup{Rogers} & \textup{Levenshtein79} & \textup{K.--L.} & \textup{Cohn--Zhao} & \textup{C.--Z.}+L79 &\textup{Our bound}\\
\hline
12 &\cellcolor{gray!10}$8.759\times 10^{-2}$ & $1.065\times 10^{-1}$ & $1.038\times 10^0$&$9.666\times 10^{-1}$&$3.253\times 10^{-1}$&$1.228\times 10^{-1}$\\
24 &\cellcolor{gray!10}$2.456\times 10^{-3}$ & $3.420\times 10^{-3}$ & $2.930\times 10^{-2}$&$2.637\times 10^{-2}$&$8.464\times 10^{-3}$&$3.194\times 10^{-3}$\\
36 &\cellcolor{gray!10}$5.527\times 10^{-5}$ & $8.109\times 10^{-5}$ & $5.547\times 10^{-4}$ &$4.951\times 10^{-4}$&$1.610\times 10^{-4}$&$6.035\times 10^{-5}$\\
48 & $1.128\times 10^{-6}$ & $1.643\times 10^{-6}$ & $8.745\times 10^{-6}$ &$7.649\times 10^{-6}$&$2.534\times 10^{-6}$&\cellcolor{gray!10}$9.487\times 10^{-7}$\\
60 & $2.173\times 10^{-8}$ & $3.009\times 10^{-8}$ &$1.223\times10^{-7}$&$1.046\times 10^{-7}$&$3.521\times 10^{-8}$&\cellcolor{gray!10}$1.317\times 10^{-8}$\\
72 & $4.039\times 10^{-10}$ & $5.135\times 10^{-10}$ &$1.550\times 10^{-9}$&$1.322\times 10^{-9}$& $4.496\times 10^{-10}$ &\cellcolor{gray!10}$1.678\times 10^{-10}$\\
84 & $7.315\times 10^{-12}$ & $8.312\times 10^{-12}$ &$1.850\times 10^{-11}$&$1.574\times 10^{-11}$&$5.381\times 10^{-12}$&\cellcolor{gray!10}$2.007\times 10^{-12}$\\
96 & $1.300\times 10^{-13}$ & $1.291\times 10^{-13}$ &$2.111\times 10^{-13}$&$1.786\times 10^{-13}$&$6.101\times 10^{-14}$&\cellcolor{gray!10}$2.273\times 10^{-14}$\\
108 & $2.277\times 10^{-15}$ & $1.937\times 10^{-15}$ &$2.320\times 10^{-15}$&$1.942\times 10^{-15}$&$6.662\times 10^{-16}$&\cellcolor{gray!10}$2.480\times 10^{-16}$\\
120 & $3.940\times 10^{-17}$ & $2.826\times 10^{-17}$ &$2.452\times 10^{-17}$&$2.051\times 10^{-17}$&$7.058\times 10^{-18}$&\cellcolor{gray!10}$2.626\times 10^{-18}$
\end{tabular}
\caption{Upper bounds on maximal sphere packing densities $\delta_n$ in $\mathbb{R}^n$ using functions that are not constructed using computer assistance. The last column is obtained using the method of this paper.}\end{table}

We use our geometrically constructed test functions to obtain the last column of Table 1. Table 1 is a comparison of upper bounds on sphere packing densities. This table does not include the \textit{computer-assisted} mathematically rigorous bounds obtained using the Cohn--Elkies linear programming method or those of semi-definite programming. We now describe the different columns. The \textit{Rogers} column corresponds to the bounds on sphere packing densities obtained by Rogers~\cite{Rogers}. The \textit{Levenshtein79} column corresponds to the bound obtained by Levenshtein in terms of roots of Bessel functions~\cite{Leven79}. The \textit{K.--L.} column corresponds to the bound on $M(n,\theta)$ proved by Kabatyanskii and Levenshtein~\cite{KL} combined with inequality~\eqref{sid}. The \textit{Cohn--Zhao} column corresponds to the column found in the work of Cohn and Zhao~\cite{CohnZhao}; they combined their inequality~\eqref{CohnZhaoineq} with the bound on $M(n,\theta)$ proved by Kabatyanskii--Levenshtein~\cite{KL}. We also include the column \textit{C.--Z.+L79} which corresponds to combining Cohn and Zhao's inequality with improved bounds on $M(n,\theta)$ using Levenshtein's optimal polynomials~\cite{Leven79}. The final column corresponds to the bounds on sphere packing densities obtained by our method. In Table 1, the highlighted entries are the best bounds obtained from these methods. Our bounds break most of the other bounds also in low dimensions. Our bounds are obtained from explicit geometrically constructed functions satisfying the Cohn--Elkies linear programming method. Our method only involves explicit integral calculations; in contrast to the numerical method in \cite{Elkies}, we do not rely on any searching algorithm. 
Moreover, compared to the Cohn--Elkies linear programming method, in $n=120$ dimensions, we improve upon the sphere packing density upper bound of $1.164\times 10^{-17}$ obtained by forcing eight double roots. 
\subsection{Structure of the paper}In Section~\ref{GM}, we setup some of the notation used in this paper and prove Proposition~\ref{gmim}. Section~\ref{newtest} concerns the general construction of our test functions that are used in conjunction with the Delsarte and Cohn--Elkies linear programming methods. In Section~\ref{Comparison}, we prove our main Theorems~\ref{mainbound} and~\ref{spherepacking}. In this section, we use our estimates on the triple density functions proved in Section~\ref{tripleden}. In Subsection~\ref{local}, we describe Jacobi polynomials, Levenshtein's optimal polynomials, and locally approximate Jacobi polynomials near their largest roots. In the final Section~\ref{numerics}, we provide a table of improvement factors.\\
\\
\textit{Acknowledgments.} Both authors are thankful to Alexander Barg, Peter Boyvalenkov, Henry Cohn, Matthew de Courcy-Ireland, Mehrdad Khani Shirkoohi, Peter Sarnak, and Hamid Zargar. N.T.Sardari's work is supported partially by the National Science Foundation under Grant No. DMS-2015305 N.T.S and is grateful to the Institute for Advanced Study and the Max Planck Institute for Mathematics in Bonn for their hospitality and financial support. M.Zargar was supported by the Max Planck Institute for Mathematics in Bonn, SFB1085 at the University of Regensburg, and the University of Southern California.

\section{Geometric improvement}\label{GM}
In this section, we prove Proposition~\ref{gmim}, improving inequality~\eqref{Agrell} by removing the restrictions on the angles for large dimensions $n$. This is achieved via a concentration of mass phenomenon, at the expense of an exponentially decaying error term. First, we introduce some notations that we use throughout this paper.\\
\\
Let $0<\theta<\theta'<\pi$ be given angles, and let $s:=\cos\theta$ and $s':=\cos\theta'$. Throughout, $S^{n-1}$ is the \textit{unit} sphere centered at the origin of $\mathbb{R}^n$. Suppose $\boldsymbol{z}\in S^{n-1}$ is a fixed point.\\
\\
Consider the hyperplane $N_{\boldsymbol{z}}:=\{\boldsymbol{w}\in\mathbb{R}^n:\left<\boldsymbol{w},\boldsymbol{z}\right>=0\}$. For each $\boldsymbol{a},\boldsymbol{b}\in S^{n-1}$ of radial angle at least $\theta$ from each other, we may orthogonally project them onto $N_{\boldsymbol{z}}$ via the map $\Pi_{\boldsymbol{z}}:S^{n-1}\setminus\{\pm\boldsymbol{z}\}\rightarrow N_{\boldsymbol{z}}$. For every $\boldsymbol{a}\in S^{n-1}$,
\[\Pi_{\boldsymbol{z}}(\boldsymbol{a})=\boldsymbol{a}-\left<\boldsymbol{a},\boldsymbol{z}\right>\boldsymbol{z}.\]
For brevity, when $\boldsymbol{a}\neq\pm\boldsymbol{z}$ we denote $\frac{\Pi_{\boldsymbol{z}}(\boldsymbol{a})}{|\Pi_{\boldsymbol{z}}(\boldsymbol{a})|}$ by $\tilde{\boldsymbol{a}}$ lying inside the unit sphere $S^{n-2}$ in $N_{\boldsymbol{z}}$ centered at the origin. Given $\boldsymbol{a},\boldsymbol{b}\in S^{n-1}\setminus\{\pm\boldsymbol{z}\}$, we obtain points $\tilde{\boldsymbol{a}},\tilde{\boldsymbol{b}}\in S^{n-2}$. We will use the following notation.
\[u:=\left<\boldsymbol{a},\boldsymbol{z}\right>,\]
\[v:=\left<\boldsymbol{b},\boldsymbol{z}\right>,\]
and
\[t:=\left<\boldsymbol{a},\boldsymbol{b}\right>.\]
It is easy to see that
\begin{equation}\label{xlab}
\left<\tilde{\boldsymbol{a}},\tilde{\boldsymbol{b}}\right>=\frac{t-uv}{\sqrt{(1-u^2)(1-v^2)}}.
\end{equation}
\begin{figure}[h!]
	\centering
		\includegraphics[width=60mm,scale=0.5]{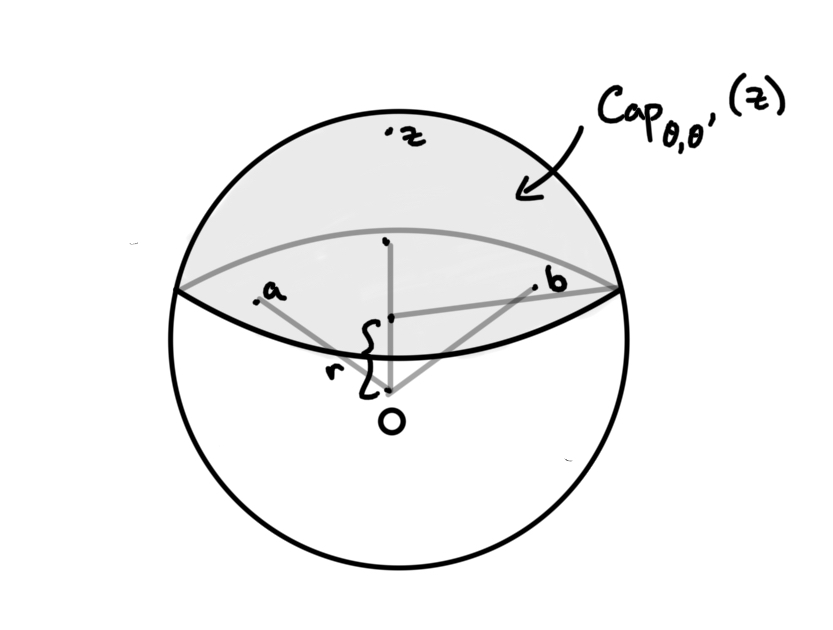}
	\label{fig:img1}
	\caption{Spherical cap $\textup{Cap}_{\theta,\theta'}(\boldsymbol{z})$ containing points $\boldsymbol{a},\boldsymbol{b}$.}\label{img4}
\end{figure}
Consider the cap $\textup{Cap}_{\theta,\theta'}(\boldsymbol{z})$ on $S^{n-1}$ centered at $\boldsymbol{z}$ and of radius $\frac{\sin(\theta/2)}{\sin(\theta'/2)}$. $\textup{Cap}_{\theta,\theta'}(\boldsymbol{z})$ is the spherical cap centered at $\boldsymbol{z}$ with the defining property that any two points $\boldsymbol{a}$, $\boldsymbol{b}$ on its boundary of radial angle $\theta$ are sent to points $\tilde{\boldsymbol{a}}$, $\tilde{\boldsymbol{b}}\in S^{n-2}$ having radial angle $\theta'$.\\
\\
It will be convenient for us to write the radius of the cap as $\sqrt{1-r^2}=\frac{\sin(\theta/2)}{\sin(\theta'/2)}$, from which it follows that
\begin{equation}\label{littler}r=\sqrt{\frac{s-s'}{1-s'}}.
\end{equation}
This $r$ is the distance from the center of the cross-section defining the cap $\textup{Cap}_{\theta,\theta'}(\boldsymbol{z})$ to the center of $S^{n-1}$. In fact, 
\[\textup{Cap}_{\theta,\theta'}(\boldsymbol{z})=\{\boldsymbol{a}\in S^{n-1}:\left<\boldsymbol{a},\boldsymbol{z}\right>\geq r\}.\]\\
\\
For $0<\theta<\theta'<\pi$, we define
\begin{equation}\label{cosgamma}\gamma_{\theta,\theta'}:=2\arctan\frac{s}{\sqrt{(1-s)(s-s')}}+\arccos(r)-\pi,
\end{equation}
and 
\begin{equation}\label{hugeR}
R:=\cos(\gamma_{\theta,\theta'}).
\end{equation}
Then $R>r$, and we define the strip
\[\textup{Str}_{\theta,\theta'}(\boldsymbol{z}):=\left\{\boldsymbol{a}\in S^{n-1}:r\leq\left<\boldsymbol{a},\boldsymbol{z}\right>\leq R\right\}.\]
\begin{figure}[h!]
	\centering
		\includegraphics[width=70mm,scale=0.5]{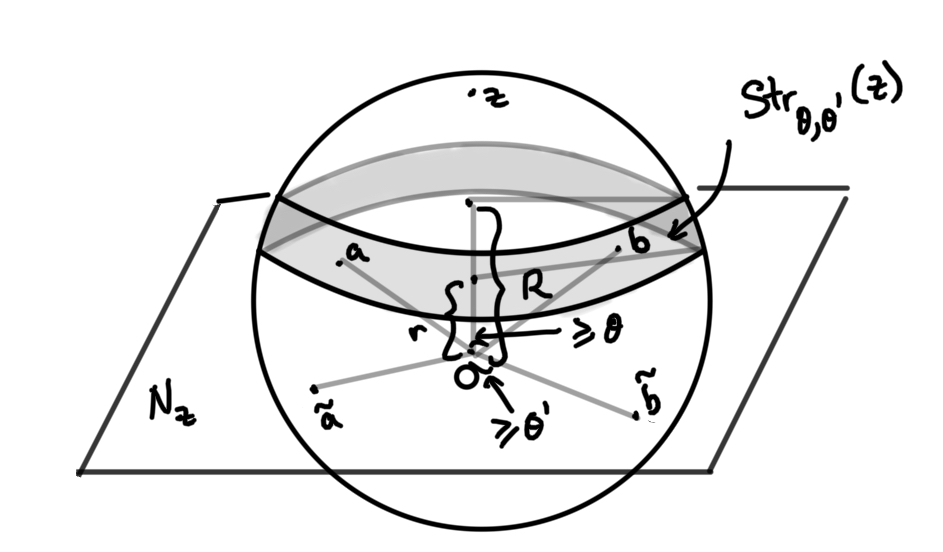}
	\label{fig:image2}
	\caption{Spherical strip $\textup{Str}_{\theta,\theta'}(\boldsymbol{z})$.}\label{img2}
\end{figure}
\begin{lemma}\label{anglelemma}
The strip $\textup{Str}_{\theta,\theta'}(\boldsymbol{z})$ is contained in $\textup{Cap}_{\theta,\theta'}(\boldsymbol{z})$. Furthermore, any two points in $\textup{Str}_{\theta,\theta'}(\boldsymbol{z})\setminus\{\pm\boldsymbol{z}\}$ of radial angle at least $\theta$ apart are mapped to points in $S^{n-2}$ (unit sphere in $N_{\boldsymbol{z}}$) of radial angle at least $\theta'$.
\end{lemma}
\begin{proof}
The first statement follows from $\left<\boldsymbol{a},\boldsymbol{z}\right>\geq r$ for any point $\boldsymbol{a}\in\textup{Str}_{\theta,\theta'}(\boldsymbol{z})$. The second statement follows from equation (147) and Lemma 42 of~\cite{Agrell}.
\end{proof}
With this in mind, we are now ready to prove Proposition~\ref{gmim}. When discussing the measure of strips $\textup{Str}_{\theta,\theta'}(\boldsymbol{z})$, we drop $\boldsymbol{z}$ from the notation and simply write $\textup{Str}_{\theta,\theta'}$.
\begin{proposition}\label{gmim} Let $0<\theta<\theta^{\prime}<\pi.$ We have 
\begin{equation}\label{geointerpolation}M(n,\theta)\leq \frac{M(n-1,\theta')}{\lambda_n(\theta,\theta^{\prime})} (1+O(ne^{-nc})),\end{equation}
where $c:=\frac{1}{2}\log\left(\frac{1-r^2}{1-R^2}\right)>0$ is independent of $n$ and only depends on $\theta$ and $\theta'$.
\end{proposition}
\begin{proof}Suppose $\{x_1,\hdots,x_N\}\subset S^{n-1}$ is a maximal $\theta$-spherical code. Given $x\in S^{n-1}$, let $m(x)$ be the number of such strips $\textup{Str}_{\theta,\theta'}(x_i)$ such that $x\in\textup{Str}_{\theta,\theta'}(x_i)$. Note that $x\in\textup{Str}_{\theta,\theta'}(x_i)$ if and only if $x_i\in\textup{Str}_{\theta,\theta'}(x)$. Therefore, the strip $\textup{Str}_{\theta,\theta'}(x)$ contains $m(x)$ points of $\{x_1,\hdots,x_N\}$. From the previous lemma, we know that these $m(x)$ points are mapped to points in $S^{n-2}$ that have pairwise radial angles at least $\theta'$. As a result,
\[m(x)\leq M(n-1,\theta'),\]
using which we obtain
\[N\cdot\mu(\textup{Str}_{\theta,\theta'})=\sum_{i=1}^N\int_{\textup{Str}_{\theta,\theta'}(x_i)}d\mu(x)=\int_{S^{n-1}}m(x)d\mu(x)\leq M(n-1,\theta')\int_{S^{n-1}}d\mu(x)=M(n-1,\theta'),\]
where $\mu$ is the uniform probability measure on $S^{n-1}$. Hence,
\begin{equation}\label{stareq}M(n,\theta)\leq\frac{M(n-1,\theta')}{\mu(\textup{Str}_{\theta,\theta'})}.\end{equation}
Note that the masses of $\textup{Str}_{\theta,\theta'}$ and the cap $\textup{Cap}_{\theta,\theta'}$ have the property that
\begin{eqnarray*}1-\frac{\mu(\textup{Str}_{\theta,\theta'})}{\lambda_n(\theta,\theta')}&=&\frac{1}{\omega_n\lambda_n(\theta,\theta')}\int_{R}^1(1-t^2)^{\frac{n-3}{2}}dt\\
&\leq& \frac{(1-R^2)^{\frac{n-3}{2}}}{\omega_n\lambda_n(\theta,\theta')}.
\end{eqnarray*}
Here, $\omega_n=\int_{-1}^1(1-t^2)^{\frac{n-3}{2}}dt$. On the other hand, we may also give a lower bound on $\lambda_n(\theta,\theta')$ by noting that
\begin{eqnarray*}
\lambda_n(\theta,\theta')&=&\frac{1}{\omega_n}\int_r^1(1-t^2)^{\frac{n-3}{2}}dt\\
&\geq &\frac{\left(\sqrt{\frac{1+r}{1-r}}\right)^{n-3}}{\omega_n}\int_r^1(1-t)^{n-3}dt\\&=&
\frac{\left(\sqrt{\frac{1+r}{1-r}}\right)^{n-3}(1-r)^{n-2}}{(n-2)\omega_n}.
\end{eqnarray*}
The first inequality follows from 
\[\frac{1+t}{1-t}\geq \frac{1+r}{1-r}\]
for $t\in[r,1)$. Combining this inequality for $\lambda_n(\theta,\theta')$ with the above, we obtain
\[1\geq \frac{\mu(\textup{Str}_{\theta,\theta'})}{\lambda_n(\theta,\theta')}\geq 1- \frac{(n-2)(1-R^2)^{\frac{n-3}{2}}}{(1-r)(1-r^2)^{\frac{n-3}{2}}}=1-\frac{(n-2)}{(1-r)}e^{-\frac{n-3}{2}\log\left(\frac{1-r^2}{1-R^2}\right)}.\]
The conclusion follows using~\eqref{stareq}.
\end{proof}
\begin{remark}\label{CZC}
Proposition~\ref{gmim} removes the angular condition on inequality~\eqref{Agrell} of Barg--Musin at the expense of an error term that is exponentially decaying in the dimension $n$. 

\end{remark}
\section{New test functions}\label{newtest}
In this section, we prove general linear programming bounds on the sizes of spherical codes and sphere packing densities by constructing new test functions. 
\subsection{Spherical codes}
Recall the definition of $D(\mu,s)$ from the discussion of the Delsarte linear programming method in the introduction. In this subsection, we construct a function inside $D(d\mu_{\frac{n-3}{2}},\cos\theta)$  from  a given one inside $D(d\mu_{\frac{n-4}{2}},\cos\theta^{\prime}),$  where $\theta'>\theta.$ \\
\\
Suppose that  $g_{\theta'}\in  D(d\mu_{\frac{n-4}{2}},\cos\theta^{\prime}).$ Fix $\vec{z}\in S^{n-1}$.   Given $\boldsymbol{a},\boldsymbol{b} \in S^{n-1}$, we define 
\begin{equation}\label{relat}
h(\boldsymbol{a},\boldsymbol{b};\vec{z}):= F(\left<\boldsymbol{a},\vec{z}\right>) F(\left<\boldsymbol{b},\vec{z}\right>  ) g_{\theta'}\left(\left<\tilde{\boldsymbol{a}},\tilde{\boldsymbol{b}}\right>\right),
\end{equation}
where $F$ is an arbitrary integrable real valued function on $[-1,1]$, and $\tilde{\boldsymbol{a}}$ and $\tilde{\boldsymbol{b}}$ are unit vectors on the tangent space of the sphere at $\vec{z}$ as defined in the previous section. Using the notation $u,v,t$ of Section~\ref{GM},
\begin{equation}\label{hpair}
h(\boldsymbol{a},\boldsymbol{b};\vec{z}):= F(u)F(v) g_{\theta'}\left(\frac{t-uv}{\sqrt{(1-u^2)(1-v^2)}}\right).
\end{equation}
Note that the projection is not defined at $\pm\boldsymbol{z}$. If either $\boldsymbol{a}$ or $\boldsymbol{b}$ is $\pm\boldsymbol{z}$, we let $h(\boldsymbol{a},\boldsymbol{b};\vec{z})=0$.
\begin{lemma}\label{51}
 $h(\boldsymbol{a},\boldsymbol{b};\vec{z})$ is a positive semi-definite function in the variables $\boldsymbol{a},\boldsymbol{b}$ on $S^{n-1},$ namely
 \[
 \sum_{x_i,x_j\in A} a_i a_j h(\boldsymbol{a}_i,\boldsymbol{a}_j;\vec{z})\geq 0
 \]
 for every finite subset $A\subset S^{n-1}$, and  coefficients $a_i \in \mathbb{R}.$ Moreover, $h$ is invariant under the diagonal action of the orthogonal group $O(n)$, namely 
 \[
 h(\boldsymbol{a},\boldsymbol{b};\vec{z})=h(k\boldsymbol{a},k\boldsymbol{b};k\vec{z})
 \]
 for every $k\in O(n)$.
\end{lemma}
\begin{proof}
By~\eqref{relat}, we have 
 \[
 \sum_{x_i,x_j\in A} a_i a_j h(\boldsymbol{a}_i,\boldsymbol{a}_j;\vec{z})= \sum_{{x_i,x_j\in A} } a_iF(\left<\vec{x_i},\vec{z}\right>) a_jF(\left<\vec{x_j},\vec{z}\right>  ) g_{\theta'}(\left<\tilde{\vec{x_i}},\tilde{\vec{x_j}}\right>).
 \]
 By~\eqref{positivedef} and the positivity of the Fourier coefficients of $g_{\theta'}$, we have 
 \[
 \sum_{{x_i,x_j\in A} } a_iF(\left<\vec{x_i},\vec{z}\right>) a_jF(\left<\vec{x_j},\vec{z}\right>  ) g_{\theta'}(\left<\tilde{\vec{x_i}},\tilde{\vec{x_j}}\right>) \geq 0.
 \]
 This proves the first part, and the second part follows from~\eqref{relat} and the $O(n)$-invariance  of the  inner product. 
\end{proof}
Let 
\begin{equation}\label{relative}
 h(\boldsymbol{a},\boldsymbol{b}):=\int_{O(n)}  h(\boldsymbol{a},\boldsymbol{b};k\vec{z}) d\mu(k),
\end{equation}
where $d\mu(k)$ is the probability Haar measure on $O(n).$
\begin{lemma}\label{pointpair}
$ h(\boldsymbol{a},\boldsymbol{b})$ is a positive semi-definite  point pair invariant function on $S^{n-1}.$
\end{lemma}
\begin{proof}
This follows from the previous lemma. 
\end{proof}
Since $h(\boldsymbol{a},\boldsymbol{b})$ is a point pair invariant function, it only depends on $t=\langle \boldsymbol{a},\boldsymbol{b}\rangle.$ For the rest of this paper, we abuse notation and consider $h$ as a real valued function of $t$ on $[-1,1],$ and write  $h(\boldsymbol{a},\boldsymbol{b})=h(t).$
\subsubsection{Computing $\cal{L}(h)$}
See~\eqref{Ldef} for the definition of $\cal{L}$. We proceed to computing the value of $\cal{L}(h)$ in terms of $F$ and $g_{\theta'}.$
 First, we compute the value of $h(1).$ Let 
  $ \| F \|^2_2:=\int_{-1}^{1} F(u)^2  d\mu_{\frac{n-3}{2}}(u).$
\begin{lemma}\label{h(1)}
We have 
\[
h(1)=g_{\theta'}(1) \| F \|^2_2.
\]
\end{lemma}
\begin{proof}
Indeed, by definition, $h(1)$ corresponds to taking $\boldsymbol{a}=\boldsymbol{b}$, from which it follows that $t=1$, $u=v$ and $\frac{t-uv}{\sqrt{(1-u^2)(1-v^2)}}=1$. Therefore, we obtain

\[
h(1)=  \int_{-1}^{1} F(u)^2 g_{\theta'}(1) d\mu_{\frac{n-3}{2}}(u)=g_{\theta'}(1) \| F \|^2_2.
\]
\end{proof}
Next, we compute the zero Fourier coefficient of $h,$ that is, $h_0:=\int_{-1}^{1} h(t)  d\mu_{\frac{n-3}{2}}(t)$. Let  $F_0=\int_{-1}^{1} F(u)  d\mu_{\frac{n-3}{2}}(u)  $ and $g_{\theta',0}=\int_{-1}^{1} g_{\theta'}  d\mu_{\frac{n-4}{2}}(u)$.
\begin{lemma}\label{h_0}
We have
\[
h_0=g_{\theta',0}F_0^2.
\]
\end{lemma}
\begin{proof}
Let $O(n-1)\subset O(n)$ be the stabilizer  of $\vec{z}.$ We identify $O(n)/O(n-1)$ with $S^{n-1}$ and write $[k_1]:= k_1\vec{z}\in S^{n-1}$ for any $k_1\in O(n).$
Then we write the probability Haar measure of $O(n)$ as the product of the probability Haar measure of $O(n-1)$ and the uniform probability measure $d\sigma$ of $S^{n-1}:$
\[
 d\mu(k_1)=d\mu(k_1')d\sigma([k_1]),
\]
where $k_1'\in O(n-1).$
By equations~\eqref{relat}, ~\eqref{relative}  and the above, we obtain
\[
\begin{split}
h_0&=\iint_{k_i'\in O(n-1)}  \iint_{[k_i]\in S^{n-1}}
F(\left<k_1'[k_1],\vec{z}\right>) F(\left<k_2'[k_2],\vec{z}\right>  ) g_{\theta'}\left(\left<\widetilde{k_1'[k_1]},\widetilde{k_2'[k_2]}\right>\right)d\mu(k_1')d\sigma([k_1])d\mu(k_2')d\sigma([k_2])
\\
&= \iint_{[k_i]\in S^{n-1}}F(\left<[k_1],\vec{z}\right>) F(\left<[k_2],\vec{z}\right>  ) d\sigma([k_1]) d\sigma([k_2]) \iint_{k_i'\in O(n-1)}g_{\theta'}\left(\left<k_1'\widetilde{[k_1]}, k_2' \widetilde{[k_2]}\right>\right)d\mu(k_1')d\mu(k_2').
\end{split}
\]
We note that 
\[
\iint_{k_i'\in O(n-1)}g_{\theta'}\left(\left<k_1'\widetilde{[k_1]}, k_2' \widetilde{[k_2]}\right>\right)d\mu(k_1')d\mu(k_2')=g_{\theta',0},
\]
and 
\[
 \iint_{[k_i]\in S^{n-1}}F(\left<[k_1],\vec{z}\right>) F(\left<[k_2],\vec{z}\right>  ) d\sigma([k_1]) d\sigma([k_2])= \int_{-1}^{1}\int_{-1}^{1} F(u)F(v)  d\mu_{\frac{n-3}{2}}(u) \mu_{\frac{n-3}{2}}(v)=F_0^2.
\]
Therefore, 
\[
h_0=g_{\theta',0}F_0^2,\]
as required. 
\end{proof}
\begin{proposition}
We have 
\[
\cal{L}(h)=\cal{L}(g_{\theta'})\frac{\| F \|^2_2}{F_0^2}.
\]
\end{proposition}
\begin{proof}
This follows immediately from Lemmas~\ref{h(1)} and~\ref{h_0}. 
\end{proof}
\subsubsection{Criterion for  $h\in D(d\mu_{\frac{n-3}{2}},\cos\theta)$}\label{liftsection}
Finally, we give a criterion which implies $h\in D(d\mu_{\frac{n-3}{2}},\cos\theta).$ 
Recall that $0<\theta< \theta',$ and $0<r<R<1$ are as defined in Section~\ref{GM}. Let $s^{\prime}=\cos(\theta^{\prime})$ and $s=\cos(\theta).$ Note that $s^{\prime}<s$,  $r=\sqrt{\frac{s-s^{\prime}}{1-s^\prime}}$, and $R=\cos\gamma_{\theta,\theta'}$ as in equation~\eqref{cosgamma}. We define
\[
\chi(y)
:=\begin{cases}
1 &\text{ for } r\leq y\leq R,
\\
0 &\text{otherwise}.
\end{cases}
\]
\begin{proposition}\label{lift}
Suppose that  $g_{\theta'}\in  D(d\mu_{\frac{n-4}{2}},\cos\theta^{\prime})$ is given and $h$ is defined as in \eqref{relative} for some  $F$.
Suppose that $F(x)$ is a positive integrable function giving rise to an $h$ such that $h(t)\leq 0$ for every $-1\leq t\leq\cos\theta$. Then
\[
h\in D(d\mu_{\frac{n-3}{2}},\cos\theta),
\]
and 
\[
 M(n,\theta)\leq \cal{L}(h)=\cal{L}(g_{\theta'})\frac{\| F \|^2_2}{F_0^2}.
\]
Among all positive integrable functions $F$ with compact support inside $[r,R]$, $\chi$ minimize the value of $ \cal{L}(h),$ and for $F=\chi$ we have 
\[
M(n,\theta)\leq \cal{L}(h)\leq \frac{ \cal{L}(g_{\theta'})}{\lambda_n(\theta,\theta')}(1+O(ne^{-nc})),
\]
where $c=\frac{1}{2}\log\left(\frac{1-r^2}{1-R^2}\right)>0.$
\end{proposition} 
\begin{proof}The first part follows from the previous lemmas and propositions. Let us specialize to the situation where $F$ is merely assumed to be a positive integrable function with compact support inside $[r,R]$. Let us first show that $h(t)\leq 0$ for $t\leq s$. We have 
 \[
  h(t):=\int_{O(n)}  h(\boldsymbol{a},\boldsymbol{b};k\vec{z}) d\mu(k),
 \]
 where $
h(\boldsymbol{a},\boldsymbol{b};\vec{z}):= F(u)F(v) g_{\theta'}\left(\frac{t-uv}{\sqrt{(1-u^2)(1-v^2)}}\right).
$
First, note that $F(u)F(v)\neq 0$ implies that $\boldsymbol{a}$ and $\boldsymbol{b}$ belong to $\textup{Str}_{\theta,\theta'}(\vec{z}).$
By Lemma~\ref{anglelemma}, the radial angle between $\tilde{\boldsymbol{a}}$ and $\tilde{\boldsymbol{b}}$ is at least $\theta'$, and so
\[\frac{t-uv}{\sqrt{(1-u^2)(1-v^2)}}=\left<\tilde{\boldsymbol{a}},\tilde{\boldsymbol{b}}\right>\in [-1,\cos\theta'].\]
Therefore
\[
g_{\theta'}\left(\frac{t-uv}{\sqrt{(1-u^2)(1-v^2)}}\right)\leq 0
\]
when $F(u)F(v)\neq 0.$ Hence, the integrand $h(\boldsymbol{a},\boldsymbol{b},\boldsymbol{z})$ is non-positive when $t\in[-1,\cos\theta]$, and so $h(t)\leq 0$ for $t\leq s$.
\\

It is easy to see that when $F=\chi$, 
\[
\|F\|_2^2=\mu(\textup{Str}_{\theta,\theta'}(\vec{z}))
\]
and 
\[
F_0=\mu(\textup{Str}_{\theta,\theta'}(\vec{z})).
\]
Therefore, by our estimate in the proof of Proposition~\ref{gmim} we have 
\[
M(n,\theta)\leq \cal{L}(h)\leq  \frac{ \cal{L}(g_{\theta'})}{\lambda_n(\theta,\theta')}\left(1+O\left(ne^{-\frac{n}{2}\log\left(\frac{1-r^2}{1-R^2}\right)}\right)\right).
\]
Finally, the optimality follows from the Cauchy--Schwarz inequality. More precisely, since $F(x)$ has compact support inside $[r,R]$, we have 
\[
\mu(\textup{Str}_{\theta,\theta'}(\vec{z})) \|F\|_2^2 \geq  F_0^2.
\]
Therefore, $\cal{L}(h)=\frac{g_{\theta'}(1) }{g_{\theta',0}}\frac{\| F \|^2_2}{F_0^2}\geq  \frac{\cal{L}(g_{\theta'})}{\mu(\textup{Str}_{\theta,\theta'}(\vec{z}))}  $ with equality only when $F=\chi.$
\end{proof}
\subsection{Sphere packings}
Suppose $0<\theta\leq\pi$ is a given angle, and suppose $g_{\theta}\in D(d\mu_{\frac{n-3}{2}},\cos\theta)$. Fixing $\boldsymbol{z}\in\mathbb{R}^n$, for each pair of points $\boldsymbol{a},\boldsymbol{b}\in\mathbb{R}^n\backslash \{\boldsymbol{z}\}$ consider
\[H(\boldsymbol{a},\boldsymbol{b};\boldsymbol{z}):=F(|\boldsymbol{a}-\boldsymbol{z}|)F(|\boldsymbol{b}-\boldsymbol{z}|)g_{\theta}\left(\left<\frac{\boldsymbol{a}-\boldsymbol{z}}{|\boldsymbol{a}-\boldsymbol{z}|},\frac{\boldsymbol{b}-\boldsymbol{z}}{|\boldsymbol{b}-\boldsymbol{z}|}\right>\right),\]
where $F$ is an even  positive function on $\mathbb{R}$ such that it is in $L^1(\mathbb{R}^n)\cap L^2(\mathbb{R}^n)$. We may then define $H(\boldsymbol{a},\boldsymbol{b})$ by averaging over all $\boldsymbol{z}\in\mathbb{R}^n$:
\begin{equation}\label{HTfunct}H(\boldsymbol{a},\boldsymbol{b}):=\int_{\mathbb{R}^n}H(\boldsymbol{a},\boldsymbol{b};\boldsymbol{z})d\boldsymbol{z}.\end{equation}
\begin{lemma}\label{57}
 $H(\boldsymbol{a},\boldsymbol{b})$ is a positive semi-definite kernel on $\mathbb{R}^n$ and depends only on $T=|\boldsymbol{a}-\boldsymbol{b}|$.
\end{lemma}
\begin{proof}
The proof is similar to that of~Lemma~\ref{51}.
\end{proof}
 As before, we abuse notation and write $H(T)$ instead of $H(\boldsymbol{a},\boldsymbol{b})$ when $T=|\boldsymbol{a}-\boldsymbol{b}|$. The analogue of Proposition~\ref{lift} is then the following.
\begin{proposition}\label{liftsphere}
Let $0<\theta\leq\pi$ and suppose $g_{\theta}\in D(d\mu_{\frac{n-3}{2}},\cos\theta)$. Suppose $F$ is as above such that $H(T)\leq 0$ for every $T\geq 1$. Then 
\begin{equation}\label{densitygeneralbound}
\delta_n\leq\frac{\vol(B_1^n)\|F\|_{L^2(\mathbb{R}^n)}^2}{2^n\|F\|_{L^1(\mathbb{R}^n)}^2}\cal{L}(g_{\theta}),
\end{equation}
where $\vol(B_1^n)$ is the volume of the $n$-dimensional unit ball. In particular, if $F=\chi_{[0,r]}$, where  $0\leq r\leq 1$, is such that it gives rise to an $H$ satisfying $H(T)\leq 0$ for every $T\geq 1$, then\begin{equation}\label{densitybound}\delta_n\leq \frac{\cal{L}(g_{\theta})}{(2r)^n}.
\end{equation} 
\end{proposition}
\begin{proof}
The proof of this proposition is similar to that of Theorem 3.4 of Cohn--Zhao~\cite{CohnZhao}.  We focus our attention on proving inequality~\eqref{densitygeneralbound}. Suppose we have a packing of $\mathbb{R}^n$ of density $\Delta$ by non-overlapping balls of radius $\frac{1}{2}$. By Theorem 3.1 of Cohn--Elkies~\cite{Elkies}, we have
\[\Delta\leq\frac{\vol(B_1^n)H(0)}{2^n\widehat{H}(0)}.\]
Note that $H(0)=g_{\theta}(1)\|F\|_{L^2(\mathbb{R}^n)}^2$, and that
\[\widehat{H}(0)=\int_{\mathbb{R}^n}H(|\boldsymbol{z}|)d\boldsymbol{z}=g_{\theta,0}\|F\|_{L^1(\mathbb{R}^n)}^2.\]
As a result, we obtain inequality~\eqref{densitygeneralbound}. The rest follows from a simple computation.
\end{proof}
Note that the situation $r=\frac{1}{2\sin(\theta/2)}$ with $\pi/3\leq\theta\leq \pi$ corresponds to Theorem 3.4 of Cohn--Zhao~\cite{CohnZhao}, as checking the negativity condition $H(T)\leq 0$ for $T\geq 1$ follows from Lemma 2.2 therein. The factor $2^n$ comes from considering functions where the negativity condition is for $T\geq 1$ instead of $T\geq 2$.
\begin{remark}
In this paper, we consider characteristic functions; however, it is an interesting open question to determine the optimal such $F$ in order to obtain the best bounds on sphere packing densities through this method. \end{remark}
\subsection{Incorporating geometric improvement into linear programming}
In proving upper bounds on $M(n,\theta)$, Levenshtein~\cite{Leven79,Levenshtein}, building on Kabatyanskii--Levenshtein~\cite{KL}, constructed feasible test functions $g_{\theta}\in D(d\mu_{\frac{n-3}{2}},\cos\theta)$ for the Delsarte linear programming problem with $\cal{L}(g_{\theta})=M_{\textup{Lev}}(n,\theta)$ defined in~\eqref{defmlev}. This gave the bound
\begin{equation}\label{KLsub}M(n,\theta)\leq M_{\textup{Lev}}(n,\theta).\end{equation}
Sidelnikov's geometric inequality~\eqref{sid} gives
\[M(n,\theta)\leq \frac{M(n+1,\theta')}{\lambda_n(\theta,\theta^{\prime})}\]
for angles $0<\theta<\theta'<\frac{\pi}{2}$. Applying this for $0<\theta<\theta^*$ and combining with inequality~\eqref{KLsub},  Kabatyanskii and Levenshtein obtained
\begin{equation}\label{KLup}
M(n,\theta)\leq \frac{M_{\textup{Lev}}(n+1,\theta^*)}{\lambda_n(\theta,\theta^*)}.
\end{equation}
From~\cite{KL}, it is known that this is exponentially better than inequality~\eqref{KLsub} for $0<\theta<\theta^*$. Finding functions $h\in  D(d\mu_{\frac{n-3}{2}},\cos\theta)$ with $\cal{L}(h)<\cal{L}(g_{\theta})$ was suggested by Levenshtein in~\cite[page 117]{Levenshtein}.  In fact, Boyvalenkov--Danev--Bumova~\cite{Danyo} gives necessary and sufficient conditions for constructing extremal \textit{polynomials} that improve Levenshtein's bound. However, their construction does not exponentially improve inequality~\eqref{KLsub} for $0<\theta<\theta^*$. In contrast to their construction, Proposition~\ref{lift} gives the following corollary stating that our construction of the function $h$ gives an exponential improvement in the linear programming problem comparing to Levenshtein's optimal polynomials for $0<\theta<\theta^*$. This is not to say that we exponentially improve the bounds given for spherical codes and sphere packings; we provide a single function using which the Delsarte linear programming method gives a better version of the exponentially better inequality~\eqref{KLup}.
\begin{corollary}\label{introcorollary}
Fix $0<\theta<\theta^*.$  Let $h_{\theta^*}\in D(d\mu_{\frac{n-3}{2}},\cos\theta)$ be the function associated to Levenshtein's $g_{n-1,\theta^*}$ constructed in Proposition~\ref{lift}. Then
\[M(n,\theta)\leq \cal{L}(h_{\theta^*})\leq e^{-n(\delta_{\theta}+o(1))}M_{\textup{Lev}}(n,\theta),\]
where $o(1)\rightarrow 0$ as $n\rightarrow\infty$ and $\delta_{\theta}:=\Delta(\theta)-\Delta(\theta^*)>0$, with
\[\Delta(\theta):=\frac{1+\sin\theta}{2\sin\theta}\log\frac{1+\sin\theta}{2\sin\theta}-\frac{1-\sin\theta}{2\sin\theta}\log\frac{1-\sin\theta}{2\sin\theta}+\frac{1}{2}\log(1-\cos\theta).\]
\end{corollary}
\begin{proof}
Let $0<\theta<\theta^{\prime}\leq\pi/2$. Recall the notation of Proposition~\ref{lift}. Associated to a function $g_{\theta'}$ satisfying the Delsarte linear programming conditions in dimension $n-1$ and for angle $\theta'$, in Proposition~\ref{lift} we constructed an explicit $h_{\theta'}\in D(d\mu_{\frac{n-3}{2}},\cos\theta)$ such that
\[\cal{L}(h_{\theta'})\leq \frac{ \cal{L}(g_{\theta'})}{\lambda_n(\theta,\theta')}(1+O(ne^{-nc})),\]
where $c>0$ is a specific constant depending only on $\theta$ and $\theta'$. Therefore,

\[\frac{1}{n}\log\cal{L}(h_{\theta'})\leq \frac{1}{n}\log\cal{L}(g_{\theta'}) - \frac{\log\lambda_n(\theta,\theta')}{n} + O(e^{-nc}).\]
By~\cite[Theorem 4]{KL} 
\[
\lim_{n\to \infty} \frac{1}{n}\log\cal{L}(g_{\theta'})=\frac{1+\sin\theta'}{2\sin\theta'}\log\frac{1+\sin\theta'}{2\sin\theta'}-\frac{1-\sin\theta'}{2\sin\theta'}\log\frac{1-\sin\theta'}{2\sin\theta'}.
\]
 It is easy to show that~\cite[Proof of Theorem 4]{KL}
 \[
 \lim_{n\to\infty}\frac{\log\lambda_n(\theta,\theta')}{n}= \frac{1}{2}\log\frac{1-\cos\theta}{1-\cos\theta'}.
 \]

 Hence,
 \[
 \frac{1}{n}\log\cal{L}(h_{\theta'})\leq \Delta(\theta')-\frac{1}{2}\log(1-\cos\theta)+o(1),
 \]
 where
 \[
 \Delta(\theta'):=\frac{1+\sin\theta'}{2\sin\theta'}\log\frac{1+\sin\theta'}{2\sin\theta'}-\frac{1-\sin\theta'}{2\sin\theta'}\log\frac{1-\sin\theta'}{2\sin\theta'}+\frac{1}{2}\log(1-\cos\theta').
 \]
 Note that 
 \[
 \frac{d}{d\theta'} \Delta(\theta')=-\frac{\csc^2(\theta')}{2}\left(\cos \theta' \log(\frac{1+\sin\theta'}{1-\sin\theta'})-(1+\cos\theta')\sin\theta'\right).
 \]
 As we mentioned before,  $\theta^*:= 62.997...^{\circ}$ is the unique root of the equation 
$\frac{d}{d\theta'} \Delta(\theta')=0$~\cite[Theorem 4]{KL} in the interval $0<\theta'<\pi/2$, which is the unique minimum of $\Delta(\theta')$ for $0<\theta'< \pi/2.$ Hence, for $0<\theta<\theta^*,$ taking $h_{\theta^*}$ the function constructed in Proposition~\ref{lift} associated to Levenshtein's optimal polynomial $g_{n-1,\theta^*}$ for angle $\theta^*$ in dimension $n-1$, and $g_{n,\theta}$ Levenshtein's optimal polynomial for angle $\theta$ in dimension $n$, we obtain
\[
\frac{1}{n} \left(\log\cal{L}(h_{\theta^*}) -\log\cal{L}(g_{n,\theta}) \right) \leq \Delta(\theta^*)-\Delta(\theta)+o(1)<0
\]
for sufficiently large $n$. This concludes the proof of this proposition.
\end{proof}
\subsection{Constant improvement to Barg--Musin and Cohn--Zhao}\label{BMCZ}
In this subsection, we sketch the ideas that go into proving Theorems~\ref{mainbound} that improves inequality~\eqref{Agrell} of Barg--Musin by a constant factor of at least $0.4325$ for every angle $\theta$ with $0<\theta<\theta^*$. The improvement to the Cohn--Zhao inequality~\eqref{CohnZhaoineq} for sphere packings is similar. We will complete the technical details in the rest of the paper.\\
\\
\begin{figure}[h!]\label{yellow}
	\centering
		\includegraphics[width=40mm,scale=0.5]{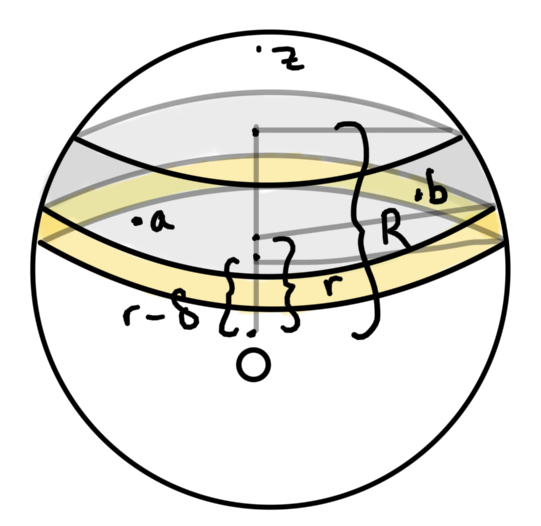}
	\label{fig:img1}
	\caption{$r$ and $R$ defined in~\eqref{littler} and~\eqref{hugeR}.}\label{img5}
\end{figure}

Recall that given any test function $g_{\theta'}\in  D(d\mu_{\frac{n-4}{2}},\cos\theta^{\prime})$ and $F$ an arbitrary integrable real valued function on $[-1,1]$, we defined
\[
h(\boldsymbol{a},\boldsymbol{b};\vec{z}):= F(u)F(v) g_{\theta'}\left(\frac{t-uv}{\sqrt{(1-u^2)(1-v^2)}}\right)
\]
in equation~\eqref{hpair}, using which we defined the function
\begin{equation}\label{hxyint}
  h(\boldsymbol{a},\boldsymbol{b}):=\int_{O(n)}  h(\boldsymbol{a},\boldsymbol{b};k\vec{z}) d\mu(k).
\end{equation}
$h(\boldsymbol{a},\boldsymbol{b})$ is a point-pair invariant function, and so we viewed it as a function $h(t)$ of $t:=\left<\boldsymbol{a},\boldsymbol{b}\right>$. As we saw in the proof of Proposition~\ref{lift}, for $F$ positive integrable and compactly supported on $[r,R]$, $F(u)F(v)\neq 0$ implies that $\boldsymbol{a},\boldsymbol{b}\in\text{Str}_{\theta,\theta'}(\vec{z})$, where
\[\textup{Str}_{\theta,\theta'}(\vec{z}):=\{\boldsymbol{a}\in S^{n-1}:r\leq\left<\boldsymbol{a},\vec{z}\right>\leq R\}\]
is the grey strip illustrated above in Figure~\ref{img5}. From this, we obtained that
\[g_{\theta'}\left(\frac{t-uv}{\sqrt{(1-u^2)(1-v^2)}}\right)\leq 0,\]
using which we obtained that the integrand in~\eqref{hxyint} is non-positive, giving us 
\begin{equation}\label{hnegt}
h(t)\leq 0\textup{ for }t\in[-1,\cos\theta].
\end{equation}
Our insight is that the integrand in~\eqref{hxyint} need not be non-positive everywhere in order to have~\eqref{hnegt}. In fact, we will show that when $F$ is the characteristic function with support $[r-\delta,R]$ for some $\delta>\frac{c}{n}$, where $c>0$ is independent of $n$, and $g_{\theta'}$ are Levenshtein's optimal polynomials, we continue to have~\eqref{hnegt}. This corresponds to allowing $\boldsymbol{a},\boldsymbol{b}$ be contained in a slight enlargement of the strip $\textup{Str}_{\theta,\theta'}(\vec{z})$ including also the yellow region in Figure~\ref{img5}.\\
\\
There are two main ingredients that go into determining an explicit lower bound for $\delta$. The first is related to understanding the behavior of Levenshtein's optimal polynomials near their largest roots. This reduces to understanding the behavior of Jacobi polynomials near their largest roots. This is done in Subsection~\ref{local}. The other idea is estimating the density function of the inner product matrix of  triple uniformly distributed points on high-dimensional spheres. This allows us to rewrite the integral~\eqref{hxyint} and its sphere packing analogue using different coordinates. This is done in Section~\ref{tripleden}.

\section{Conditional density functions}\label{tripleden}
In this section, we rewrite the averaging integral~\eqref{hxyint} when $F$ is a characteristic function in different coordinates. We also do the same for sphere packings. See equations~\eqref{htdensity} and~\eqref{Htdensity}. In order to prove Theorems~\ref{mainbound} and~\ref{spherepacking} in Section~\ref{Comparison}, we will also need to estimate certain conditional densities, which is the main purpose of this section.

\subsection{Conditional density for spherical codes}
Let $\mathcal{S}$ be the space of positive semi-definite symmetric $3\times 3$ matrices with $1$ on diagonal. Let 
\[\pi: (S^{n-1})^3 \rightarrow \mathcal{S}\] be the map that sends triple points to their pairwise inner products via
\[(\boldsymbol{a},\boldsymbol{b},\vec{z})\mapsto \begin{bmatrix}   1& t& u \\ t &1 & v \\ u &v &1\end{bmatrix},\]
where $(t,u,v):=(\left<\boldsymbol{a},\boldsymbol{b}\right>,\left<\boldsymbol{a},\vec{z}\right>,\left<\boldsymbol{b},\vec{z}\right>)$. Let $\mu(u,v,t)dudvdt$ be the density function of  
the pushforward of the product of uniform probability measures on $(S^{n-1})^3 $ to the coordinates $(u,v,t)$.

\begin{proposition}
We have
\[
\mu(u,v,t):= C\det\begin{bmatrix}   1& t& u \\ t &1 & v \\ u &v &1\end{bmatrix}^{\frac{n-4}{2}}.
\]
where $C$ is a normalization constant such that $\int_{\mathcal{S}} \mu(u,v,t) dudvdt=1$
\end{proposition}
\begin{proof}
The density of the conditional measure is given by
\[
\mu(u,v,t)=\mu(u,v;t)\mu(t)
\]
where $\mu(t)=(1-t^2)^{\frac{n-3}{2}}$ and $\mu(u,v;t)$ is the conditional density of $u,v$ given $t$. We note that this conditional density function is proportional to 
\[
\lim_{\varepsilon\to 0}\frac{\text{vol}(Str(u,v,\varepsilon;\boldsymbol{a},\boldsymbol{b}))}{\varepsilon^2}.
\]
where 
\[
Str(u,v,\varepsilon;\boldsymbol{a},\boldsymbol{b}):=\{\vec{z}\in S^{n-1}: 0\leq \left<\boldsymbol{a},\vec{z}\right>-u,  \left<\boldsymbol{b},\vec{z}\right>-v \leq \varepsilon   \}.
\]
We write $\vec{z}=\vec{z}^{\perp}+\vec{z}^{\|}$, where $\vec{z}^{\|}$ is the projection of $\vec{z}$ onto the two dimensional plane spanned by $\boldsymbol{a}$ and $\boldsymbol{b}$ and $\vec{z}^{\perp}$ is orthogonal to $\boldsymbol{a}$ and $\boldsymbol{b}$. Fix $\boldsymbol{a}, \boldsymbol{b}$ and  consider the following  map from $S^{n-1}$
\[
\vec{z}\to [\vec{z}^{\|},\vec{z}^{\perp}].
\]
 The above  map has Jacobian $\frac{\det\begin{bmatrix}   1& t& u \\ t &1 & v \\ u &v &1\end{bmatrix}^{1/2}}{\sqrt{1-t^2}}$ with respect to the Euclidean metric. 
We note that the geometric locus of  $\vec{z}^{\|}$ is a rhombus  with area $\frac{\varepsilon^2}{\sqrt{1-t^2}}$.
Moreover, given $\vec{z}^{||}$, the geometric locus of $\vec{z}^{\perp}$ is a sphere of dimension $n-3$ and radius
 \[
 |\vec{z}^{\perp}|=\frac{\det\begin{bmatrix}   1& t& u \\ t &1 & v \\ u &v &1\end{bmatrix}^{1/2}}{\sqrt{1-t^2}}+O(\varepsilon).
\]

Therefore, 
\[
\text{vol}(Str(u,v,\varepsilon;x,y))=\frac{\varepsilon^2}{\sqrt{1-t^2}}\frac{\sqrt{1-t^2}}{\det\begin{bmatrix}   1& t& u \\ t &1 & v \\ u &v &1\end{bmatrix}^{1/2}}\left( \frac{\det\begin{bmatrix}   1& t& u \\ t &1 & v \\ u &v &1\end{bmatrix}^{1/2}}{\sqrt{1-t^2}}\right)^{n-3}(1+O(\varepsilon)),
\]
from which it follows that 
\[
\mu(u,v;t)=C \frac{\det\begin{bmatrix}   1& t& u \\ t &1 & v \\ u &v &1\end{bmatrix}^{\frac{n-4}{2}}}{(1-t^2)^{\frac{n-3}{2}}}
\]
for some constant $C>0$.
This implies our proposition.
 
\end{proof}

Recall that $\theta<\theta^{\prime}$. Let $s^{\prime}=\cos(\theta^{\prime})$ and $s=\cos(\theta).$ Note that $s^{\prime}<s$ and for $r=\sqrt{\frac{s-s^{\prime}}{1-s^\prime}}$, we have $s^{\prime}=\frac{s-r^2}{1-r^2}$ and $0< r<1$. Let $0<\delta=o(\frac{1}{\sqrt{n}})$ that we specify later, and define
\[
\chi(y)
:=\begin{cases}
1 &\text{ for } r-\delta \leq y\leq R,
\\
0 &\text{otherwise}.
\end{cases}
\]
See Figure~\ref{yellow}. $\boldsymbol{a},\boldsymbol{b}$ are in the shaded areas (both grey and yellow) correspond to $u,v$ being in the support of $\chi$. 
 Recall~\eqref{xlab}, and denote 
\[
x:=\frac{t-uv}{\sqrt{(1-u^2)(1-v^2)}}.
\]
Let $\mu(x;t,\chi)$ be  the induced density function on $x$ subjected to the conditions of fixed $t,$ and $r-\delta\leq u,v\leq R$. Precisely, up to a positive constant multiple that depends on $n,t, R, r-\delta$, we have
\[
\mu(x;t,\chi)=\int_{C_{x,t}}\chi(u)\chi(v)\mu^*(x,l,t)dl,
\]
where the integral is over the curve $C_{x,t}\subset \mathbb{R}^2$ that is given by $\frac{t-uv}{\sqrt{(1-u^2)(1-v^2)}}=x$, $dl$ is the induced Euclidean metric on $C_{x,t},$ and 
\[
\mu^*(x,l,t)dxdldt= \mu(u,v,t)dudvdt.
\]
\begin{figure}[h!]
	\centering
		\includegraphics[width=50mm,scale=0.5]{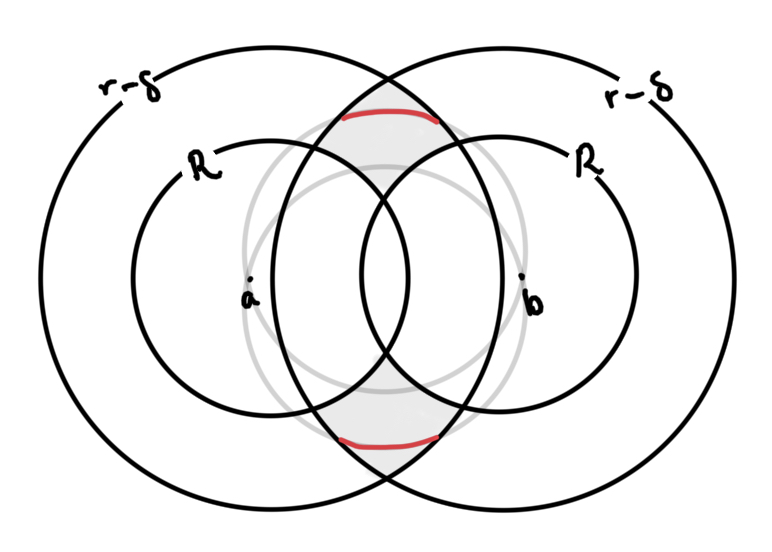}
	\caption{$\boldsymbol{a}$ and $\boldsymbol{b}$ are two points with fixed $t=\left<\boldsymbol{a},\boldsymbol{b}\right>$. The red curve is the $2$-dimensional representation of the locus of points $\boldsymbol{z}$ having $(u,v)$ on the curve $C_{x,t}$. Shaded region corresponds to the support of $\chi(u)\chi(v)$.}\label{curve}
\end{figure}

We explicitly compute $\mu^*(x,l,t).$ We have  
\[
du dv dt= \frac{1}{\sqrt{(\frac{\partial x}{\partial u})^2+(\frac{\partial x}{\partial v} )^2}} dx dl dt,
\]
from which it follows that
\[\mu^*(x,l,t)=\frac{\mu(u,v,t)}{\sqrt{(\frac{\partial x}{\partial u})^2+(\frac{\partial x}{\partial v} )^2}}.\]
Hence,  up to a positive constant multiple that depends on $n,t, R, r-\delta$, 
\[
\mu(x;t,\chi)=\int_{C_{x,t}}\frac{\chi(u)\chi(v)\mu(u,v,t)}{\sqrt{(\frac{\partial x}{\partial u})^2+(\frac{\partial x}{\partial v} )^2}} dl.
\]
We may write the test function constructed in equation~\eqref{relative} with $F=\chi$ and any given $g$  as
\begin{equation}\label{htdensity}
h(t)=\int_{-1}^1g(x)\mu(x;t,\chi)dx,
\end{equation}
where $t=\left<\boldsymbol{a},\boldsymbol{b}\right>$; see Subsection~\ref{BMCZ}. We define for \textit{complex} $x$
\[x^{+}:=\begin{cases}x &\text{ for } x\geq 0, \\
 0   &\text{otherwise.}\end{cases}
 \]
\begin{proposition}\label{triples}
Suppose that $|x-s^{\prime}|=o(\frac{1}{\sqrt{n}}).$
We have, up to a positive constant multiple depending on $n,s,R,r-\delta$,
\[
\mu(x;s,\chi)=\left(\frac{2(1-r^2)^2}{r(1-s)}+o(1)\right)\left(\delta+\sqrt{\frac{s-x}{1-x}}-r  \right)^+  \left(\left(\frac{1-x^2}{x^2}\right) \left(s-r^2\right)^2\right)^{\frac{n-4}{2}}e^{\left(-\frac{2nr\left(\sqrt{\frac{s-x}{1-x}}-r \right)}{ s-r^2}  \right)}.
\]
\end{proposition}
\begin{proof}
Let $C_x:=C_{x,s}$. Note that
\[x=\frac{s-uv}{\sqrt{(1-u^2)(1-v^2)}}\] 
and
\[
\mu(u,v,s)=\left((1-x^2)(1-u^2)(1-v^2)\right)^{\frac{n-4}{2}}.
\]
Furthermore,
\[\left(\frac{\partial x}{\partial u}\right)^2+\left(\frac{\partial x}{\partial v}\right)^2=\frac{(us-v)^2}{(1-v^2)(1-u^2)^3}+\frac{(vs-u)^2}{(1-u^2)(1-v^2)^3}=2\frac{r^2(1-s)^2}{(1-r^2)^4}+o(1).\]
Hence, up to a positive constant multiple depending on $n,s,R,r-\delta$,
\begin{eqnarray*}
\mu(x;s,\chi)&=&\left(\frac{(1-r^2)^2}{\sqrt{2}r(1-s)}+o(1)\right)\int_{C_x}\chi(u)\chi(v)\mu(u,v,s) dl
\\
&=&\left(\frac{(1-r^2)^2}{\sqrt{2}r(1-s)}+o(1)\right)(1-x^2)^{\frac{n-4}{2}}\int_{C_x}\chi(u)\chi(v)\left((1-u^2)(1-v^2)\right)^{\frac{n-4}{2}} dl
\\
&=&\left(\frac{(1-r^2)^2}{\sqrt{2}r(1-s)}+o(1)\right)(1-x^2)^{\frac{n-4}{2}}\int_{C_x}\chi(u)\chi(v)\left(\frac{s-uv}{x}\right)^{{n-4}} dl
\\
&=&\left(\frac{(1-r^2)^2}{\sqrt{2}r(1-s)}+o(1)\right)\left(\frac{1-x^2}{x^2}\right)^{\frac{n-4}{2}}\int_{C_x}\chi(u)\chi(v)\left(s-uv\right)^{{n-4}} dl
\end{eqnarray*}
Suppose that $r\leq u,v\leq R$. Then, by the definition of $R$, the equality
\[s'=\frac{s-uv}{\sqrt{(1-u^2)(1-v^2)}}\]
occurs when $(u,v)\in\{(r,R),(R,r),(r,r)\}$. Furthermore, when $|x-s'|=o\left(\frac{1}{\sqrt{n}}\right)$, we must have $(u,v)\in\{(r,R),(R,r),(r,r)\}$ up to $o\left(\frac{1}{\sqrt{n}}\right)$. Since the integrand $\chi(u)\chi(v)(s-uv)^{n-4}$ of the last integral above is exponentially larger when $u,v=r+o\left(\frac{1}{\sqrt{n}}\right)$, the main contribution of the integral comes when $u$ and $v$ are near $r$ up to $o\left(\frac{1}{\sqrt{n}}\right)$. Writing $u=r+\tilde{u}$ and $v=r+\tilde{v},$ where $\tilde{u}, \tilde{v}=o(\frac{1}{\sqrt{n}}),$ we have
\[
s-uv=s-r^2-r(\tilde{u}+\tilde{v})-\tilde{u}\tilde{v}= (s-r^2)\left(1-\frac{r(\tilde{u}+\tilde{v})+\tilde{u}\tilde{v}}{ s-r^2}\right).
\]
Hence, up to a positive constant multiple depending on $n,s,R,r-\delta$,
\[
\mu(x;s,\chi)=\left(\frac{(1-r^2)^2}{\sqrt{2}r(1-s)}+o(1)\right)\left(\left(\frac{1-x^2}{x^2}\right) \left( s-r^2\right)^2\right)^{\frac{n-4}{2}}\int_{C_x}\chi(u)\chi(v)\left(1-\frac{r(\tilde{u}+\tilde{v})+\tilde{u}\tilde{v}}{ s-r^2}\right)^{{n-4}} dl.
\]
Recall the following inequalities, which follow easily from the Taylor expansion of $\log(1+x)$
\[
\begin{split}
e^{a-\frac{a^2}{n}} \leq (1+\frac{a}{n})^n\leq e^a
\end{split}
\]
for $|a|\leq n/2.$ We apply the above inequalities  to estimate  the integral, and obtain
\[
\int_{C_x}\chi(u)\chi(v)\left(1-\frac{r(\tilde{u}+\tilde{v})+\tilde{u}\tilde{v}}{ s-r^2}\right)^{{n-4}} dl=\left(1+o(1)\right) \int_{C_x}e^{\left(-\frac{nr(\tilde{u}+\tilde{v})}{ s-r^2}  \right)}\chi(u)\chi(v) dl.
\]
We approximate the curve $C_x$ with  the following  line 
 \[\tilde{u}+\tilde{v}= 2\left(\sqrt{\frac{s-x}{1-x}}-r \right)+o\left(\frac{1}{n}\right).\]
It follows that
\[
\int_{C_x}e^{\left(-\frac{nr(\tilde{u}+\tilde{v})}{ s-r^2}  \right)}\chi(u)\chi(v) dl =(1+o(1))2\sqrt{2}\left(\delta+\sqrt{\frac{s-x}{1-x}}-r  \right)^+e^{\left(-\frac{2nr\left(\sqrt{\frac{s-x}{1-x}}-r \right)}{ s-r^2}  \right)},
\]
from which the conclusion follows.

\end{proof}
\subsection{Conditional density for sphere packings}
Let  $s=\cos(\theta)$ and $r=\frac{1}{\sqrt{2(1-s)}}$, where $\frac{1}{3}\leq s\leq \frac{1}{2}$. Let $0<\delta=\frac{c_1}{n}$ for some fixed $c_1>0$ that we specify later, and define 
\[
\chi(y)
:=\begin{cases}
1 &\text{ for } 0 \leq y\leq r+\delta,
\\
0 &\text{otherwise}.
\end{cases}
\]
Let $\boldsymbol{a},\boldsymbol{b}$  be two randomly independently chosen points on $\mathbb{R}^{n}$ with respect to the Euclidean measure such that $|\boldsymbol{a}|,|\boldsymbol{b}|\leq r+\delta,$ where $|.|$ is the Euclidean norm. Let
\[U:=|\boldsymbol{a}|,\]
\[V:=|\boldsymbol{b}|,\]
\[T:=|\boldsymbol{a}-\boldsymbol{b}|,\]
and $\alpha$ be the angle between $\boldsymbol{a}$ and $\boldsymbol{b}$.
The pushforward of the product measure on $\mathbb{R}^{n}\times\mathbb{R}^{n}$ onto the coordinates $(U,V,\alpha)$ is, up to a positive scalar depending only on $n$, the measure  
\[
\mu(U,V,\alpha) dUdVd\alpha=U^{n-1}V^{n-1}\sin(\alpha)^{n-3}dUdVd\alpha.
\]
Let 
\begin{equation}\label{cosin}
x:=\cos \alpha=\frac{U^2+V^2-T^2}{2UV},
\end{equation}
which follows from the cosine law. We have
\[
\sin \alpha d\alpha = \frac{T}{UV}dT- \frac{\partial \frac{U^2+V^2-T^2}{2UV}}{\partial U} dU - \frac{\partial \frac{U^2+V^2-T^2}{2UV}}{\partial V} dV .
\]
Hence, 
the pushforward of the product measure on $\mathbb{R}^{n}\times\mathbb{R}^{n}$ onto the coordinates $(U,V,T)\in\mathbb{R}^3$ has the following density function up to a positive scalar depending only on $n$:
\[
\mu(U,V,T)= (U^2V^2T)\Delta(U,V,T)^{n-4},
\]
where $\Delta(U,V,T)$ is the Euclidean area of the triangle with sides $U,V,T$. If no such triangle exists, then $\Delta(U,V,T)=0$.
 Let $\mu(x;T,\chi)$ be  the induced density function on $x$ subjected to the conditions of fixed $T,$ and $  U,V\leq r+\delta$. Precisely, we have, up to a positive constant multiple depending on $n, r+\delta$ and $T$, that
\[
\mu(x;T,\chi):=\int_{C_{x,T}}\chi(U)\chi(V)\mu(x,l,T)dl,
\]
where the integral is over the curve $C_{x,T}\subset \mathbb{R}^2$ that is given by $\frac{U^2+V^2-T^2}{2UV}=x$ and $dl$ is the induced Euclidean metric on $C_{x,T},$ and 
\[
\mu(x,l,T) dxdldT= \mu(U,V,T)dUdVdT.
\]

\begin{figure}[h!]
	\centering
		\includegraphics[width=50mm,scale=0.5]{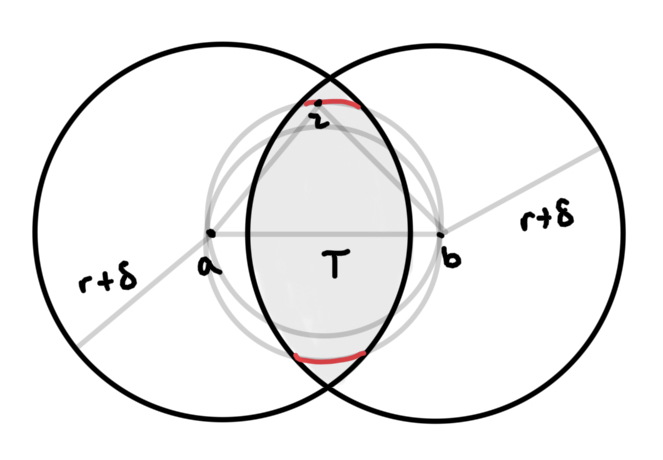}
	\caption{$\boldsymbol{a}$ and $\boldsymbol{b}$ are two points with $T=|\boldsymbol{a}-\boldsymbol{b}|$. The red curve is the locus of points $\boldsymbol{z}$ having $(U,V)$ on the curve $C_{x,T}$, that is, $\boldsymbol{z}$ with respect to which $\boldsymbol{a}$ and $\boldsymbol{b}$ have angle whose cosine is $x$. Shaded region corresponds to the support of $\chi(U)\chi(V)$.}\label{curve2}
\end{figure}

We have  
\[
dU dV dT= \frac{1}{\sqrt{(\frac{\partial x}{\partial U})^2+(\frac{\partial x}{\partial V} )^2}} dx dl dT,
\]
and
\[\left(\frac{\partial x}{\partial U}\right)^2+\left(\frac{\partial x}{\partial V} \right)^2= \left(\frac{T}{UV}\right)^2(1+x^2)-\frac{2x(1-x^2)}{UV}.\]
Hence,
\[
\mu(x;T,\chi)=\int_{C_{x,T}}\frac{\chi(U)\chi(V)\mu(U,V,T)}{\sqrt{\left(\frac{T}{UV}\right)^2(1+x^2)-\frac{2x(1-x^2)}{UV}}} dl
\]
up to a positive constant multiple depending on $n, r+\delta$ and $T$.\\
\\
From Lemma~\ref{57}, equation~\eqref{HTfunct} with $F=\chi$ and any given $g$ may be written as
\begin{equation}\label{Htdensity}
H(T)=\int_{-1}^1g(x)\mu(x;T,\chi)dx.
\end{equation}
We now study the scaling property of $\mu(x;T,\chi).$ Let $\chi_{T}(x):=\chi(\frac{x}{T}).$
\begin{lemma}\label{scalinglemma}We have
\[
\mu(x;T,\chi)=T^{2n-1}\mu(x;1,\chi_{T}).
\]
\end{lemma}
 \begin{proof}
 Note that 
 \[
 \mu(U,V,T)= T^{2n-3}\mu\left(\frac{U}{T},\frac{V}{T},1\right),
 \]
 and $x$ is invariant by scaling  $U,V,T.$ The conclusion follows. 
 \end{proof}

 Let $\mu(x;\chi):=\mu(x;1,\chi),$ and 
 \[
 r_1:=\sqrt{1-(1-x^2)(r+\delta)^2}+x(r+\delta).
 \]
\begin{proposition}\label{triple0}
Suppose that $|x-s|\leq \frac{c_2}{n}$, $x<1/2$ and $\delta=\frac{c_1}{n}$.
We have
\[
\mu(x;\chi)=(1+E)(1-x^2)^{\frac{n-4}{2}} \left( \frac{1}{2(1-x)}\right)^{{n-1}} \frac{1}{\sqrt{1-x}}\sqrt{1+\left(x-\frac{(1-x^2)r}{\sqrt{1-(1-x^2)r^2}}\right)^2}\left(r+\delta- r_1  \right)^+
\]
up to a positive scalar multiple making this a probability measure on $[-1,1]$, and where $E$ is a function of $x$ with the uniform bound $|E|<\frac{(4c_2+2c_1+2)^2}{n}$ for $n\geq 2000$.
\end{proposition}
\begin{proof}
We have 
\[
\mu(x;\chi)=\int_{C_x}\chi(U)\chi(V)\frac{\mu(U,V,1)}{{\sqrt{\frac{1+x^2}{\left(UV\right)^2}-\frac{2x(1-x^2)}{UV}}} } dl,
\]
where  $C_x:=C_{x,1}.$  We note that up to a positive scalar depending only on $n$,
\[
\mu(U,V,1)=U^2V^2\left((1-x^2)U^2V^2\right)^{\frac{n-4}{2}}.
\]
Hence,
\[
\begin{split}
\mu(x;\chi)=(1-x^2)^{\frac{n-4}{2}}\int_{C_x}\chi(U)\chi(V) \frac{\left(UV\right)^{n-1}}{{\sqrt{1+x^2-2x(1-x^2)UV}}} dl.
\end{split}
\]
Suppose that $U$ and $V$ are in the support of $\chi.$  By concentration of mass, we may assume that $U=r+\tilde{U}$ and $V=r+\tilde{V},$  with  $\tilde{U}, \tilde{V}\leq \frac{c_1}{n}.$ From~\eqref{cosin}, 
we have
 \[\tilde{U}+\tilde{V}= \frac{1}{2r(1-x)}-r+E_1,\] 
where 
\[
E_1:= \frac{1}{2r(1-x)}\left(2x\tilde{U}\tilde{V} -\tilde{U}^2-\tilde{V}^2  \right).
\]
Since $|x-s|\leq \frac{c_2}{n},$ and $\frac{1}{2}<2r(1-x)^2,$ it follows that  $-\frac{c_1+2c_2}{n}< \tilde{U}, \tilde{V}\leq \frac{c_1}{n}$ for $n\geq 2000$.  Hence, 
\[
|E_1|\leq 3\frac{(c_1+2c_2)^2}{n^2}.
\]
We have
\[
UV=r^2+r(\tilde{U}+\tilde{V})+\tilde{U}\tilde{V}= r^2\left(1+\frac{\tilde{U}+\tilde{V}}{r}\right)+\tilde{U}\tilde{V}=  \frac{1}{2(1-x)}+E_2,
\]
where $|E_2|< 4\frac{(c_1+2c_2)^2}{n^2}.$ Furthermore,  
\[
\sqrt{1+x^2-2x(1-x^2)UV}= \sqrt{1-x-2x(1-x^2)E_2}=\sqrt{1-x}+E_2',
\]
where $|E_2'|<|E_2|.$ Hence,
\[
\begin{split}
\mu(x;\chi)&=(1-x^2)^{\frac{n-4}{2}} \left( \frac{1}{2(1-x)}\right)^{{n-1}}  \frac{1}{\sqrt{1-x}+E_2'} \left( 1+\frac{E_2}{2(1-x)}\right)^{n-1} \int_{C_x} \chi(U)\chi(V) dl.
\end{split}
\]
Note that 
\[
 \frac{1}{\sqrt{1-x}+E_2'} \left( 1+\frac{E_2}{2(1-x)}\right)^{n-1} =\frac{1}{\sqrt{1-x}}(1+E_2''),
\]
where $E_2''\leq  4\frac{(c_1+2c_2)^2}{n}$ for $n\geq 2000$. We parametrize the curve $C_x$ with $V$ to obtain
\[
U(V)=\sqrt{1-(1-x^2)V^2}+xV.
\]
We have 
\[
\frac{dU}{dV}=x-\frac{(1-x^2)V}{\sqrt{1-(1-x^2)V^2}}=x-\frac{(1-x^2)r}{\sqrt{1-(1-x^2)r^2}}+E_3,
\]
where 
\[|E_3|=\left|(V-r) \frac{(1-x^2)}{(1-(1-x^2)V_1^2)^{3/2}}\right| \] for some $V_1\in\left(r-\frac{c_1+2c_2}{n},r+\frac{c_1}{n}\right),$ which implies  $\left|\frac{(1-x^2)}{(1-(1-x^2)V_1^2)^{3/2}}\right|<6.$ Hence,
\[
|E_3|\leq 6\left(\frac{c_1+2c_2}{n}\right).
\]
Hence,
\[
\begin{split}
\int_{C_x} \chi(U)\chi(V) dl&=^+\int_{r_1}^{r+\delta}\sqrt{1+\left(\frac{dU}{dV}\right)^2} dV
\\&=\sqrt{1+\left(x-\frac{(1-x^2)r}{\sqrt{1-(1-x^2)r^2}}\right)^2}\left(r+\delta- r_1  \right)^+ (1+\bar{E_3})
\end{split}
\]
where 
\[
r_1:=\sqrt{1-(1-x^2)(r+\delta)^2}+x(r+\delta),
\]
and 
\[ |E_3|\leq 6\left(\frac{c_1+2c_2}{n}\right).\]
The first equality $=^+$ above means equal to $0$ if $r+\delta<r_1$. Therefore,
\[
\mu(x;\chi)=(1+E)(1-x^2)^{\frac{n-4}{2}} \left( \frac{1}{2(1-x)}\right)^{{n-1}}\frac{1}{\sqrt{1-x}} \sqrt{1+\left(x-\frac{(1-x^2)r}{\sqrt{1-(1-x^2)r^2}}\right)^2}\left(r+\delta- r_1  \right)^+
\]
up to a positive scalar multiple, where $|E|<\frac{(4c_2+2c_1+2)^2}{n}$ for $n\geq 2000$.
This completes the proof of our Proposition.
\end{proof}

\section{Comparison with previous bounds}\label{Comparison}
We define Jacobi polynomials, state some of their properties, and prove a local approximation result for Jacobi polynomials in Subsection~\ref{local}. In Subsection~\ref{improvingsphericalcodes}, we improve bounds on $\theta$-spherical codes. In Subsection~\ref{improvingspherepackings}, we improve upper bounds on sphere packing densities. The general constructions of our test functions were provided in Section~\ref{newtest}. For each of our main theorems, we determine the largest values of $\delta=O(1/n)$ measuring the extent to which the supports of the characteristic functions $\chi$ in Propositions~\ref{lift} and~\ref{liftsphere} could be enlarged. See the end of Section~\ref{newtest} for an outline of the general strategy.\\

\subsection{Jacobi polynomials and their local approximation}\label{local}
Recall  definition \eqref{defmlev} of $M_{\textup{Lev}}(n,\theta)$. Levenshtein proved  inequality~\eqref{levinq}  by applying Delsarte's linear programming bound to a family of even and odd degree polynomials inside $D(d\mu_{\frac{n-3}{2}},\cos\theta),$ which we now discuss. In order to define these Levenshtein polynomials, we record some well-known properties of Jacobi polynomials (see~\cite[Chapter IV]{Gabor}) that we will also use in the rest of the paper.\\
\\
We denote by $p^{\alpha,\beta}_{d}(t)$ Jacobi polynomials of degree $d$ with parameters $\alpha$ and $\beta$. These are orthogonal polynomials with respect to the probability measure 
\[d\mu_{\alpha,\beta}:=\frac{(1-t)^{\alpha}(1+t)^{\beta}dt}{\int_{-1}^1(1-t)^{\alpha}(1+t)^{\beta}dt}\] on the interval $[-1,1]$ with the normalization that gives
\[p^{\alpha,\beta}_d(1)=\binom{d+\alpha}{d}.\]
 $p^{\alpha,\beta}_{d}(t)$ has $n$ simple real roots $t_{1,d}^{\alpha,\beta} > t_{2,d}^{\alpha,\beta}> \dots >t_{d,d}^{\alpha,\beta}.$ When $\alpha=\beta$, we denote the measure $d\mu_{\alpha,\alpha}$ simply as $d\mu_{\alpha}$. 
\begin{equation}\label{der}
\frac{d}{dt}p_d^{\alpha,\beta}(t)=\frac{d+\alpha+\beta+1}{2}p_{d-1}^{\alpha+1,\beta+1}(t).
\end{equation}
When proving our local approximation result on Levenshtein's optimal polynomials in the rest of this subsection, we use the fact that the Jacobi polynomial $p_d^{\alpha,\beta}(t)$ satisfies the differential equation
\begin{equation}\label{diffeq}(1-t^2)x''(t)+(\beta-\alpha-(\alpha+\beta+2)t)x'(t)+d(d+\alpha+\beta+1)x(t)=0.
\end{equation}
By~\cite[Lemma 5.89]{Levenshtein},
\[
t^{\alpha+1,\alpha+1}_{1,d-1}<t^{\alpha+1,\alpha}_{1,d}<t^{\alpha+1,\alpha+1}_{1,d}.
\]
It is also well-known that for fixed $\alpha$, $t^{\alpha+1,\alpha+1}_{1,d}\rightarrow 1$ as $d\rightarrow\infty$. Henceforth, let $d:=d(n,\theta)$ be uniquely determined by $t^{\alpha+1,\alpha+1}_{1,d-1}< \cos(\theta) \leq t^{\alpha+1,\alpha+1}_{1,d}$.
Let~\cite[Lemma 5.38]{Levenshtein}
\begin{equation}\label{defg}
g_{n,\theta}(x)=\begin{cases}
\frac{(x+1)^2}{(x-t^{\alpha+1,\alpha+1}_{1,d})}\left(  p^{\alpha+1,\alpha+1}_{d}(x) \right)^2  &\text{ if } t^{\alpha+1,\alpha}_{1,d}< \cos(\theta) \leq t^{\alpha+1,\alpha+1}_{1,d} ,
\\
\frac{(x+1)}{(x-t^{\alpha+1,\alpha}_{1,d})}\left(  p^{\alpha+1,\alpha}_{d}(x) \right)^2  &\text{ if } t^{\alpha+1,\alpha+1}_{1,d-1}< \cos(\theta) \leq t^{\alpha+1,\alpha}_{1,d},
\end{cases}
\end{equation}
where $\alpha:=\frac{n-3}{2}$ in our case. Levenshtein proved that $g_{n,\theta}\in D(d\mu_{\frac{n-3}{2}},\cos\theta),$ and 
\[
\cal{L}(g_{n,\theta})=M_{\textup{Lev}}(n,\theta).
\] 
By~\eqref{Dbound}, this gives 
\[
M(n,\theta)\leq M_{\textup{Lev}}(n,\theta).
\] 
As part of our proofs of our main theorems in the next subsections, we need to determine local approximations to Jacobi polynomials $p_d^{\alpha,\beta}$ in the neighbourhood of points $s\in (-1,1)$ such that $s\geq t_{1,d}^{\alpha,\beta}$. This is obtained using the behaviour of the zeros of Jacobi polynomials. Using this, we obtain suitable local approximations of Levenshtein's optimal functions near $s$.

\begin{proposition}\label{prop:linapproximation}Suppose  that $\alpha\geq \beta\geq 0,$ $|\alpha-\beta|\leq 1$, $d\geq0$ and    $s\in [t_{1,d}^{\alpha,\beta},1)$. Then, we have
\[p_d^{\alpha,\beta}(t)=p_d^{\alpha,\beta}(s)+(t-s)\frac{dp_d^{\alpha,\beta}}{dt}(s)(1+A(t)),\]
where,
\[|A(t)|\leq \frac{e^{\sigma(t)}-1}{\sigma(t)}-1\]
with
\[\sigma(t):=\frac{|t-s|\left(2\alpha s+2s+1\right)}{1-s^2}.\]
\end{proposition}
\begin{proof}
Consider the Taylor expansion
\[p_d^{\alpha,\beta}(t)=\sum_{k=0}^{\infty}\frac{(t-s)^k}{k!}\frac{d^kp_d^{\alpha,\beta}}{dt^k}(s)\]
of $p_d^{\alpha,\beta}$ centered at $s$. We prove the proposition by showing that for $s\in [t_{1,d}^{\alpha,\beta},1)$, the higher degree terms in the Taylor expansion are small in comparison to the linear term. Indeed, suppose $k\geq 1$. Then, using equation~\eqref{der},
\[\frac{(d^{k+1}/dt^{k+1})p_d^{\alpha,\beta}(s)}{(d^k/dt^k)p_d^{\alpha,\beta}(s)}=\frac{(d/dt)p_{d-k}^{\alpha+k,\beta+k}(s)}{p_{d-k}^{\alpha+k,\beta+k}(s)}=\sum_{i=1}^{d-k}\frac{1}{s-t_{i,d-k}^{\alpha+k,\beta+k}}\leq \sum_{i=1}^{d-1}\frac{1}{s-t_{i,d-1}^{\alpha+1,\beta+1}},\]
where the last inequality follows from the fact that the roots of a Jacobi polynomial interlace with those of its derivative. However, the last quantity is equal to $\frac{(d^2/dt^2)p_d^{\alpha,\beta}(s)}{(d/dt)p_d^{\alpha,\beta}(s)}$. We proceed to show that
\begin{equation}\label{prelimineq}\frac{(d^2/dt^2)p_d^{\alpha,\beta}(s)}{(d/dt)p_d^{\alpha,\beta}(s)}\leq \frac{2\alpha s+2s+1}{1-s^2}.
\end{equation}
Indeed, we know from the differential equation~\eqref{diffeq} that
\begin{equation}(1-s^2)(d^2/dt^2)p_d^{\alpha,\beta}(s)+(\beta-\alpha-(\alpha+\beta+2)s)(d/dt)p_d^{\alpha,\beta}(s)+d(d+\alpha+\beta+1)p_d^{\alpha,\beta}(s)=0.
\end{equation}
However, since $s$ is to the right of the largest root of $p_d^{\alpha,\beta}$, $p_d^{\alpha,\beta}(s)\geq 0$. Therefore,
\[(1-s^2)(d^2/dt^2)p_d^{\alpha,\beta}(s)+(\beta-\alpha-(\alpha+\beta+2)s)(d/dt)p_d^{\alpha,\beta}(s)\leq 0,\]
from which inequality~\eqref{prelimineq} follows. As a result,  the degree $k+1$ term compares to the linear term as
\begin{eqnarray*}\frac{(d^{k+1}/dt^{k+1})p_d^{\alpha,\beta}(s)}{(d/dt)p_d^{\alpha,\beta}(s)}\frac{|t-s|^k}{(k+1)!}&\leq& \frac{|t-s|^k}{(k+1)!}\left( \frac{2(\alpha+1)s+1}{1-s^2}\right)^k\end{eqnarray*}
Consequently, for every $k\geq 1$,
\[
\frac{|t-s|^{k+1}}{(k+1)!}(d^{k+1}/dt^{k+1})p_d^{\alpha,\beta}(s)\leq |t-s|(d/dt)p_d^{\alpha,\beta}(s)\frac{\left( \frac{|t-s|\left(2\alpha+2s+1\right)}{1-s^2}\right)^k}{(k+1)!}
\]
As a result, we obtain that
\[p_d^{\alpha,\beta}(t)=p_d^{\alpha,\beta}(s)+(t-s)\frac{dp_d^{\alpha,\beta}}{dt}(s)(1+A(t)),\]
where
\[|A(t)|\leq \frac{e^{\sigma(t)}-1}{\sigma(t)}-1\]
with
\[\sigma(t)=\frac{|t-s|\left(2\alpha s+2s+1\right)}{1-s^2}.\]
\end{proof}

\subsection{Improving spherical codes bound}\label{improvingsphericalcodes}
We prove a stronger version of Theorem~\ref{mainbound} that we now state.
  
\begin{theorem} \label{Thm1general}Fix $\theta<\theta^*$ and suppose $0<\theta<\theta'\leq\pi/2$. Then there is a function $h\in D(d\mu_{\frac{n-3}{2}},\cos\theta)$ such that 
\[\cal{L}(h) \leq  c_n \frac{M_{\textup{Lev}}(n-1,\theta')}{\lambda_n(\theta,\theta')},\]
where $c_n\leq 0.4325$ for large enough $n$ independent of $\theta$ and $\theta'.$
\end{theorem}
\begin{proof}Without loss of generality, we may assume that  $\cos(\theta^{\prime})=t^{\alpha+1,\alpha+\varepsilon}_{1,d}$ for some $\varepsilon\in\{0,1\}$ and $\alpha:=\frac{n-4}{2}.$ Indeed, recall from Subsection~\ref{local} that $d$ is uniquely determined by $t^{\alpha+1,\alpha+1}_{1,d-1}< \cos(\theta') \leq t^{\alpha+1,\alpha+1}_{1,d}$,  
\begin{equation*}
g(x)=\begin{cases}
\frac{(x+1)^2}{(x-t^{\alpha+1,\alpha+1}_{1,d})}\left(  p^{\alpha+1,\alpha+1}_{d}(x) \right)^2  &\text{ if } t^{\alpha+1,\alpha}_{1,d}< \cos(\theta') \leq t^{\alpha+1,\alpha+1}_{1,d} ,
\\
\frac{(x+1)}{(x-t^{\alpha+1,\alpha}_{1,d})}\left(  p^{\alpha+1,\alpha}_{d}(x) \right)^2  &\text{ if } t^{\alpha+1,\alpha+1}_{1,d-1}< \cos(\theta') \leq t^{\alpha+1,\alpha}_{1,d},
\end{cases}
\end{equation*}
and
\[
M_{\textup{Lev}}(n-1,\theta') = \cal{L}(g)=M_{\textup{Lev}}(n-1,\arccos(t^{\alpha+1,\alpha+\varepsilon}_{1,d})). 
\]
We also have
\[
\lambda_n(\theta,\theta') \leq \lambda_n(\theta,\arccos(t^{\alpha+1,\alpha+\varepsilon}_{1,d}) ),
\]
where the above follows from $\cos(\theta^{\prime})\leq t^{\alpha+1,\alpha+\varepsilon}_{1,d}$ and $\lambda_n(\theta,\theta^{\prime})$ is the ratio of volume of the spherical cap with radius $\frac{\sin(\theta/2)}{\sin(\theta^{\prime}/2)}$ on the unit sphere $S^{n-1}$ to the volume of the whole sphere. Hence,
\[
\frac{M_{\textup{Lev}}(n-1,\arccos(t^{\alpha+1,\alpha+\varepsilon}_{1,d}))}{\lambda_n(\theta,\arccos(t^{\alpha+1,\alpha+\varepsilon}_{1,d}))} \leq \frac{M_{\textup{Lev}}(n-1,\theta')}{\lambda_n(\theta,\theta')}.
\]

As before, $s=\cos(\theta),$ and $s'=\cos(\theta')=t^{\alpha+1,\alpha+\varepsilon}_{1,d}.$ Note that $s^{\prime}<s$ and for $r=\sqrt{\frac{s-s^{\prime}}{1-s^\prime}}$, we have $s^{\prime}=\frac{s-r^2}{1-r^2}$ and $0< r<1$. Let $0<\delta=O(\frac{1}{n})$ that we specify later, and define the function $F$ for the application of Proposition~\ref{lift} to be
\[
F(y)=\chi(y)
:=\begin{cases}
1 &\text{ for } r-\delta <y\leq R,
\\
0 &\text{otherwise}.
\end{cases}
\]

Recall from~\eqref{htdensity} that
\[h(t):=\int_{-1}^1g(x)\mu(x,t;\chi)dx.\]
 Applying Proposition~\ref{lift} with the function $F$ as above and the function $g$, we obtain the inequality in the statement of the theorem with $c_n\leq 1+o(1)$. We now prove the desired bound on $c_n$ for sufficiently large $n$.\\
\\
By Proposition~\ref{introcorollary}, for any $\delta_1>0$, if $|s'-\cos(\theta^*)|\geq \delta_1$, then for large $n$, $M_{\textup{Lev}}(n-1,\theta')/\lambda_n(\theta,\theta')$ is exponentially worse than $M_{\textup{Lev}}(n-1,\theta^*)/\lambda_n(\theta,\theta^*)$. Therefore, we assume that $s'=\cos(\theta')=t_{1,d}^{\alpha+1,\alpha+\varepsilon}$ for some $\varepsilon\in\{0,1\}$, and that $s'$ is sufficiently close to $\cos\theta^*$. We specify the precision of the difference later. Suppose $\varepsilon=0$, and so $s'=t^{\alpha+1,\alpha}_{1,d}$. By Proposition~\ref{prop:linapproximation}, we have
\[g(x)=(x+1)(x-s^{\prime})\frac{dp_d^{\alpha+1,\alpha}}{dt}(s^{\prime})^2(1+A(x))^2,\]
where
\begin{equation}\label{ax}|A(x)|\leq \frac{e^{\sigma(x)}-1}{\sigma(x)}-1\end{equation}
with
\[\sigma(x):=\frac{|x-s'|(ns'+s'+1)}{1-s'^2}.\]
By Proposition~\ref{triples}, we have for $|x-s'|=o\left(\frac{1}{\sqrt{n}}\right)$ the estimate
\[
\mu(x;s,\chi)=\left(\frac{2(1-r^2)^2}{r(1-s)}+o(1)\right)\left(\delta+\sqrt{\frac{s-x}{1-x}}-r  \right)^+  \left(\left(\frac{1-x^2}{x^2}\right) \left(s-r^2\right)^2\right)^{\frac{n-4}{2}}e^{\left(-\frac{2nr\left(\sqrt{\frac{s-x}{1-x}}-r \right)}{ s-r^2}  \right)}.
\]
We need to find the maximal $\delta>0$ such that 
\[
h(t)=\int_{-1}^1 g(x)\mu(x;t,\chi) dx \leq 0
\]
for every $-1\leq t\leq s$. We first address the above inequality for $t=s$. Note that the integrand is negative for $x<s^{\prime}$ and positive for $x>s^{\prime}.$ Hence,
\begin{eqnarray*}
\int_{-1}^1 g(x)\mu(x;s,\chi)dx=
\left|\int_{s'}^1 g(x)\mu(x;s,\chi)dx\right| - \left|\int_{-1}^{s'} g(x)\mu(x;s,\chi)dx\right|.
\end{eqnarray*}
Its non-positivity is equivalent to
\[\left|\int_{-1}^{s'} g(x)\mu(x;s,\chi)dx\right|\geq\left|\int_{s'}^1 g(x)\mu(x;s,\chi)dx\right|.\]
We proceed to give a lower bound on the absolute value of the integral over $-1\leq x\leq s^{\prime}.$ Later, we give an upper bound on the right hand side. By \eqref{ax}, we have 
\[
|1+A(x)|\geq  \left(2-\frac{e^{\sigma(x)}-1}{\sigma(x)}\right)^+.
\]

We note that $2-\frac{e^{\sigma}-1}{\sigma}$ is a concave function with value 1 as $\sigma\rightarrow 0$ and a root at $\sigma=1.256431...$. Hence,
for $\sigma<1.25643$, we have 
\[
|1+A(x)|\geq   \left(2-\frac{e^{\sigma(x)}-1}{\sigma(x)}\right).
\]
Note that $\sigma(x)<1.25643$ implies 
\[
|x-s'| \leq \frac{(1.25643)(1-s'^2)}{ns'}.
\]
Let 
\[
\lambda:=\frac{(1.25643)(1-s'^2) }{ns'}.
\]
Therefore, as $n\rightarrow\infty$
\begin{align}\label{lowerint}
& \frac{r(1-s)}{2(1-r^2)^2}\left( s-r^2\right)^{-(n-4)} \frac{dp_d^{(\alpha+1,\alpha)}}{dt}(s^{\prime})^{-2}\left|\int_{-1}^{s'} g(x)\mu(x;s,\chi)dx\right|
\\&
\gtrsim  
\int_{s'-\lambda}^{s'}(1+x) (s^{\prime}-x)\left(2-\frac{e^{\sigma(x)}-1}{\sigma(x)}\right)^2\left(\delta+\sqrt{\frac{s-x}{1-x}}-r  \right)  \left(\frac{1-x^2}{x^2}\right)^{\frac{n-4}{2}}e^{\left(-\frac{2nr\left(\sqrt{\frac{s-x}{1-x}}-r \right)}{ s-r^2}  \right)}dx. \nonumber
\end{align}
We change the variable $s'-x$ to $z$ and note that
\begin{equation}\label{taylor1}
\sqrt{\frac{s-x}{1-x}}-r =\frac{z(1-s) }{2(s-s')^{1/2}(1-s')^{3/2}}+O(\lambda^2)
\end{equation}
\begin{equation}\label{taylor2}
\frac{1-x^2}{x^2}= \frac{1-s'^2}{s'^2}\left(1+\frac{2z}{s'(1-s'^2)}+ O(\lambda^2)\right)
\end{equation}
for $|x-s'|<\lambda.$
Hence, we obtain that as $n\rightarrow\infty$ and $|x-s'|<\lambda$,
\begin{eqnarray*}
e^{\left(-\frac{2nr\left(\sqrt{\frac{s-x}{1-x}}-r \right)}{ s-r^2}  \right)}&=& e^{\left(-\frac{2nr\left(\frac{z(1-s) }{2(s-s')^{1/2}(1-s')^{3/2}} \right)}{ s-r^2}  \right)+O(\lambda)}=e^{\left(\frac{-nz}{s'(1-s')}  \right)+O(\lambda)}
\\
 \left(\frac{1-x^2}{x^2}\right)^{\frac{n-4}{2}}&=&\left(\frac{1-s'^2}{s'^2}\right)^{\frac{n-4}{2}} e^{\left(\frac{nz}{s'(1-s'^2)} \right)+O(\lambda)}
 \\
 2-\frac{e^{\sigma(x)}-1}{\sigma(x)}&=&  \left(2-\frac{e^{\frac{nzs'}{(1-s'^2)}}-1}{\frac{nzs'}{(1-s'^2)}}\right)(1+O(\lambda)).
 \end{eqnarray*}
for $|x-s'|<\lambda.$ We replace the above asymptotic formulas and obtain that as $n\rightarrow\infty$, the right hand side of \eqref{lowerint} is at least
\[
(1+s')\left(\frac{1-s'^2}{s'^2}\right)^{\frac{n-4}{2}} 
\int_{0}^{\lambda}z \left(2-\frac{e^{\frac{nzs'}{(1-s'^2)}}-1}{\frac{nzs'}{(1-s'^2)}} \right)^2\left(\delta+\frac{z(1-s) }{2(s-s')^{1/2}(1-s')^{3/2}} \right) 
e^{\left(-\frac{nz}{1-s'^2}\right)} dz.
\]
We now give an upper bound on the absolute value of the integral over $s'\leq x\leq 1.$  We note that 
\[
\sqrt{\frac{s-x}{1-x}}-r =\frac{z(1-s) }{2(s-s')^{1/2}(1-s')^{3/2}}+O(\lambda^2).
\]
Let $\Lambda:= \frac{2(s-s')^{1/2}(1-s')^{3/2}\delta}{(1-s)}$. We have
\[\left(\delta+\sqrt{\frac{s-x}{1-x}}-r  \right)^+=0\]
for $x-s'>\Lambda.$ We have
\[|1+A(x)|\leq \frac{e^{\sigma(x)}-1}{\sigma(x)},\]
where
\[\sigma(x)=\frac{n|x-s'|s'}{1-s^{'2}}+o(1).\]
Therefore,
\begin{eqnarray*}
&&\frac{r(1-s)}{2(1-r^2)^2}\left( s-r^2\right)^{-(n-4)} \frac{dp_d^{(\alpha+1,\alpha)}}{dt}(s^{\prime})^{-2}\left|\int_{s'}^{1} g(x)\mu(x;s,\chi)dx\right| 
\\
&\lesssim& (1+s')  \left(\frac{1-s'^2}{s'^2}\right)^{\frac{n-4}{2}} 
\int_{0}^{\Lambda}z\left(\frac{e^{\frac{nzs'}{(1-s'^2)}}-1}{\frac{nzs'}{(1-s'^2)}} \right)^2\left(\delta-\frac{z(1-s) }{2(s-s')^{1/2}(1-s')^{3/2}} \right) 
e^{\left(\frac{nz}{1-s'^2} \right)} dz.
\end{eqnarray*}
We choose $\delta$ such that 
\[
\begin{split}
\int_{0}^{\lambda}z \left(2-\frac{e^{\frac{nzs'}{(1-s'^2)}}-1}{\frac{nzs'}{(1-s'^2)}} \right)^2\left(\delta+\frac{z(1-s) }{2(s-s')^{1/2}(1-s')^{3/2}} \right) 
e^{\left(-\frac{nz}{1-s'^2} \right)} dz
\\
 \geq  \int_{0}^{\Lambda}z\left(\frac{e^{\frac{nzs'}{(1-s'^2)}}-1}{\frac{nzs'}{(1-s'^2)}} \right)^2\left(\delta-\frac{z(1-s) }{2(s-s')^{1/2}(1-s')^{3/2}} \right) 
e^{\left(\frac{nz}{1-s'^2} \right)} dz.
\end{split}
\]
For large enough $n$ we may replace the numerical value $\cos(1.0995124)$ for $s'$ as we may assume that $s'$ is close to $\cos(\theta^*)$. Furthermore, write $v:=nz$, and divide by $\delta$ to  obtain 
\[
\begin{split}
\int_{0}^{2.196823}v \left(2-\frac{e^{0.571931v}-1}{0.571931v} \right)^2\left(1+\frac{v}{n\Lambda} \right) 
e^{\left(-1.259674v \right)} dv
\\
 \geq  \int_{0}^{n\Lambda}v\left(\frac{e^{0.571931v}-1}{0.571931v}\right)^2\left(1-\frac{v}{n\Lambda} \right) 
e^{\left(1.259674v\right)} dv.
\end{split}
\]
Here, we have also used that $\Lambda=\frac{2\delta(s-s')^{1/2}(1-s')^{3/2}}{(1-s) }$. Also, note that $n\lambda=2.196823...$ when $s'$ is near $\cos(1.0995124)$. 
We have the Taylor expansion around $v=0$
\begin{eqnarray*}&&v \left(\frac{e^{0.571931v}-1}{0.571931v} \right)^2e^{\left(1.259674v \right)}\\&=&v + 1.83161 v^2 + 1.70465 v^3 + 1.07403 v^4 + 0.514959 v^5 + 0.200242 v^6 + 0.0657225 v^7 \\&+& 0.0187113 v^8 + 0.00471321 v^9 + 0.0010662 v^{10} + 0.000219143 v^{11} + 0.0000413083 v^{12} + Er\end{eqnarray*}
with error $|Er|<2\times 10^{-5}$ if $n\Lambda<0.92$, which we assume to be the case. Simplifying, we want to find the maximal $n\Lambda$ such that
\begin{eqnarray*}
&&0.195783+\frac{0.140655}{n\Lambda}\geq\int_{0}^{n\Lambda}v\left(\frac{e^{0.571556v}-1}{0.571556v}\right)^2\left(1-\frac{v}{n\Lambda} \right) 
e^{\left(1.259392v\right)} dv
\\&=&3.17756\times 10^{-6} (n\Lambda)^2 \Bigg((n\Lambda)^{11} + 5.74715 (n\Lambda)^{10} + 30.5037 (n\Lambda)^9 + 148.328 (n\Lambda)^8 + 654.286 (n\Lambda)^7 \\&+& 2585.41 (n\Lambda)^6 + 9002.5 (n\Lambda)^5 + 27010.2 (n\Lambda)^4 + 67600.9 (n\Lambda)^3 + 134116(n\Lambda)^2 + 192140(n\Lambda) + 157353\Bigg)\\
&-&\frac{1}{n\Lambda}\left(-0.360653 + 2.42691 e^{1.25967(n\Lambda)} - 3.33819 e^{1.83161(n\Lambda)} + 1.27193 e^{2.40354(n\Lambda)}\right)+Er,
\end{eqnarray*}
where the error $Er$ again satisfies $|Er|\leq 2\times 10^{-5}$. A numerical computation gives us that the inequality is satisfied when $n\Lambda\leq 0.915451...$. Consequently, if we choose $\delta=\ell/n$, then we must have
\[\ell\leq 0.915451...\frac{(1-s)}{2(s-s')^{1/2}(1-s')^{3/2}}.\]
In this case, the cap of radius $\sqrt{1-r^2}$ becomes $\sqrt{1-(r-\delta)^2}=\sqrt{1-r^2}\left(1+\frac{\ell r}{n(1-r^2)}\right)+O(1/n^2)$. Note that $r=\sqrt{\frac{s-s'}{1-s'}}$, and so
\begin{eqnarray*}
\frac{r}{1-r^2}=\frac{(s-s')^{1/2}(1-s')^{1/2}}{1-s}.
\end{eqnarray*}
We deduce that,
\[\frac{\ell r}{(1-r^2)}\leq 0.915451...\frac{(1-s)}{2(s-s')^{1/2}(1-s')^{3/2}}\cdot \frac{(s-s')^{1/2}(1-s')^{1/2}}{1-s}=\frac{0.457896862...}{1-s'}=0.83837237...\]
This computation shows that for sufficiently large $n$, for any $0\leq\delta\leq \frac{0.83837237(1-r^2)}{nr}$, $h(s)\leq 0$. We now show that $h(t)\leq 0$ for $-1\leq t<s$. Note that
\[r(t):=\sqrt{\frac{(t-s^{\prime})^+}{1-s^\prime}}<\sqrt{\frac{s-s^{\prime}}{1-s^\prime}}=r.\]
If $s-t$ is of order greater than $O(1/n)$, then $r(t)<r-\delta$ for large $n$, and so $h(t)\leq 0$ as the integrand is negative in this case. Therefore, we may assume that $-1\leq t<s$ and $s-t=O(1/n)$. For such $t$, replacing $s$ with $t$ in the above calculations shows that the negativity $h(t)\leq 0$ is true for all
$0\leq\delta\leq \frac{0.83837237(1-r(t)^2)}{nr(t)}.$ Since $r(t)<r$, $\frac{0.83837237(1-r(t)^2)}{nr(t)}>\frac{0.83837237(1-r^2)}{nr}$. Consequently, $h(t)\leq 0$ for every $-1\leq t\leq s$ whenever $0\leq \delta\leq \frac{0.83837237(1-r^2)}{nr}$.\\
\\
Similarly, one  obtains the same conclusion when $\varepsilon=1$, that is  $s^{\prime}=t^{\alpha+1,\alpha+1}_{1,d}$. Therefore, our improvement to Levenshtein's bound on $M(n,\theta)$ for large $n$ is by a factor of $1/e^{0.83837237...}=0.432413...$ for any choice of angle $0<\theta<\theta^*$. As the error in our computations is less than $2\times 10^{-5}$, we deduce that we have an improvement by a factor of $0.4325$ for sufficiently large $n$.
\end{proof}

\subsection{Improving bound on sphere packing densities}\label{improvingspherepackings}
In this subsection, we give our improvement to  Cohn and Zhao's~\cite{CohnZhao} bound on sphere packings. Recall that in the case of sphere packings, we let $s=\cos(\theta)$, $r=\frac{1}{\sqrt{2(1-s)}}$, and $\alpha=\frac{n-3}{2}$. By assumption, $\frac{1}{3}\leq s\leq \frac{1}{2}$. In this case, we define for each $0<\delta=\frac{c_1}{n}$ the function $F$ to be used in Proposition~\ref{liftsphere} to be
\[
F(y):=\chi(y)
:=\begin{cases}
1 &\text{ for } 0 \leq y\leq r+\delta,
\\
0 &\text{otherwise}.
\end{cases}
\]
\begin{proof}[Proof of Theorem~\ref{spherepacking}]As in the proof of Theorem~\ref{Thm1general}, we consider $g=g_{n,\theta}\in D(d\mu_{\frac{n-3}{2}},\theta)$. Note that by Proposition~\ref{liftsphere} applied to the function $F$ above and $g$, we have for the function
\[H(T)=\int_{-1}^1g(x)\mu(x;T,\chi)dx\]
of equation~\eqref{Htdensity}, the inequality
\[\delta_{n}\leq \frac{\cal{L}(g)}{(2(r+\delta))^{n}}=\frac{M_{\textup{Lev}}(n,\theta)}{(2(r+\delta))^{n}}\]
if $\delta>0$ is chosen such that $H(T)\leq 0$ for $T\geq 1$. We wish to maximize $\delta$ under this negativity condition. Let $n\geq 2000$ and $s$ be a root of the Jacobi polynomial as before. As in the proof of Theorem~\ref{mainbound}, we begin by considering that case where $\varepsilon=0$, that is, $s=t^{\alpha+1,\alpha}_{1,d}$, and take $g$ for this $s$. We show that $H(1)\leq 0$. For the other $T$, showing $H(T)\leq 0$ follows similarly by using Lemma~\ref{scalinglemma}.
Note that
\begin{eqnarray*}
\int_{-1}^1 g(x)\mu(x;\chi)dx=\left|\int_{s}^1 g(x)\mu(x;\chi)dx\right| - \left|\int_{-1}^{s} g(x)\mu(x;\chi)dx\right|.
\end{eqnarray*}
To show $H(1)\leq 0$, it suffices to show that
\begin{equation}\label{ineq0}
\left|\int_{-1}^{s} g(x)\mu(x;\chi)dx\right|\geq \left|\int_{s}^1 g(x)\mu(x;\chi)dx\right|.
\end{equation}
First, we give a lower bound for the left hand side of equation~\eqref{ineq0}. As before, Proposition~\ref{prop:linapproximation} allows us to write
\[g(x)=(x+1)(x-s^{\prime})\frac{dp_d^{\alpha+1,\alpha}}{dt}(s^{\prime})^2(1+A(x))^2,\]
where
\begin{equation*}|A(x)|\leq \frac{e^{\sigma(x)}-1}{\sigma(x)}-1\end{equation*}
with
\[\sigma(x):=\frac{|x-s|(ns-s+1)}{1-s^2}.\]
 
As before, 
\[
|1+A(x)|\geq   \left(2-\frac{e^{\sigma(x)}-1}{\sigma(x)}\right)^+.
\]
Note that if $\sigma(x)<1.25643$, ensuring that the right hand side of the bound on $1+A(x)$ is non-negative, then 
\[
|x-s| \leq \frac{1.25643(1-s^{2}) }{ns+1}.
\]
Let 
\[
\lambda:=\frac{1.25643(1-s^{2}) }{ns+1}.
\]
Let
\[r_1:=\sqrt{1-(1-x^2)(r+\delta)^2}+x(r+\delta).\]

Suppose that $|x-s|\leq \frac{c_2}{n}$ and $\delta=\frac{c_1}{n}$ for some $0<c_1<0.81$, $0<c_2<3.36$. Note that for such a $c_1$ and $n\geq 2000$, the assumption $s\leq 1/2$ implies that $r+\delta\leq 1$. By Proposition~\ref{triple0}, we have $\mu(x;\chi)=0$ for $r_1\geq r+\delta.$ Otherwise,
\[
\mu(x;\chi)=C_n(1+E)(1-x^2)^{\frac{n-4}{2}} \left( \frac{1}{2(1-x)}\right)^{{n-1}} \frac{1}{\sqrt{1-x}}\sqrt{1+\left(x-\frac{(1-x^2)r}{\sqrt{1-(1-x^2)r^2}}\right)^2}\left(r+\delta- r_1  \right)
\]
for some constant $C_n>0$ making $\mu(x;\chi)$ a probability measure on $[-1,1]$, and where  $|E|<\frac{(4c_2+2c_1+2)^2}{n}$ for $n\geq 2000$. In particular, for such $c_1,c_2$, we have $|E|<\frac{292}{n}$. By letting $z:=s-x$ and $v:=\frac{(ns+1)z}{(1-s^{2})}$, Taylor expansion approximations imply that inequality~\eqref{ineq0} is satisfied for $\frac{1}{3}\leq s\leq \frac{1}{2}$ and large $n$ if
\begin{equation}
\int_0^{1.25643}v\left(2-\frac{e^v-1}{v}\right)^2e^{-3v}\left(1+\frac{3v}{2c_1}\right)dv\geq \int_0^{\frac{2c_1}{3}}v\left(\frac{e^v-1}{v}\right)^2e^{3v}\left(1-\frac{3v}{2c_1}\right)dv.
\end{equation}
By a numerics similar to that done for spherical codes, one finds that the maximal such $c_1<1$ is $0.66413470...$. Therefore, for every $\frac{1}{3}\leq s\leq \frac{1}{2}$ and $n\rightarrow\infty$, we have an improvement at least as good as
\[e^{-0.6641347/r}= \exp(-0.6641347\sqrt{2(1-s)}).\]
Note that for such $s$, as $n\rightarrow\infty$, this gives us an improvement of at most $0.5148$, a universal such improvement factor.\\
\\
Returning to the case of general $n\geq 2000$ and $\frac{1}{3}\leq s\leq \frac{1}{2}$, given such an $n$ we need to maximize $c_1<0.81$ such that

\begin{eqnarray*}
&&\left(1-\frac{312}{n}\right)\int_0^{1.25643}v\left(2-\frac{e^v-1}{v}\right)^2e^{-3v}\left(1+\frac{1.499v}{c_1}-\frac{17.31}{c_1n}\right)dv\\
&\geq&\left(1+\frac{312}{n}\right)\int_0^{0.667c_1}v\left(\frac{e^v-1}{v}\right)^2e^{3v}\left(1-\frac{1.499v}{c_1}+\frac{17.31}{c_1n}\right)dv.
\end{eqnarray*}

By a numerical calculation with \texttt{Sage}, we obtain that the improvement factor for any $\frac{1}{3}\leq s\leq \frac{1}{2}$ and any $n\geq 2000$ is at least as good as
\[0.515+74/n.\]
On the other hand, if we fix $s$ such that $s$ is sufficiently close to $s^*=\cos(\theta^*)$, then the same kind of calculations as above give us an asymptotic improvement constant of $0.4325$, the same as in the case of spherical codes. In fact, for $n\geq 2000$, we have an improvement factor at least as good as
\[0.4325+51/n\]
over the combination of Cohn--Zhao~\cite{CohnZhao} and Levenshtein's optimal polynomials~\cite{Leven79}.\\
\\
The case $s=t^{\alpha+1,\alpha+1}_{1,d}$ follows in exactly the same way. This completes the proof of our theorem.
\end{proof}
\begin{remark}
We end this section by saying that our improvements above are based on a \textit{local} understanding of Levenshtein's optimal polynomials, and that there is a loss in our computations. By doing numerics, we may do computations without having to rely on such local approximations.
\end{remark}
\section{Numerics}\label{numerics}
 In Theorem~\ref{spherepacking}, we improved  the sphere packing densities by a factor of $0.4325$ for sufficiently high dimensions. There is a loss in our estimates due to neglecting the  contribution of Levenshtein's polynomials away from their largest roots, as giving rigorous estimates is difficult.  In this section, we numerically investigate the behavior of our constant improvement factors for sphere packing densities by considering those neglected terms in dimensions up to $130$. As we noted before in the introduction, in low dimensions, there are better bounds on sphere packing densities using semi-definite programming, and so the objective of this section is guessing the improvement over $0.4325$ in high dimensions.\\

Before our work, the best known upper bound on sphere packing densities in high dimensions was obtained using inequalities
\[
\delta_n\leq\sin^n(\theta/2)\textup{Lev}(n,\theta),
\]
where $\text{Lev}(n,\theta)$ is the linear programming bound using Levenshtein's optimal polynomials~\cite[eq.(3),(4)]{Leven79}. Note that 
\(
\text{Lev}(n,\theta)\leq M_{\textup{Lev}}(n,\theta)
\)
and equality occurs when $\cos(\theta)=t^{\alpha+1,\alpha+\varepsilon}_{1,d}.$ In this section, we apply Proposition~\ref{liftsphere} to Levenshtein's optimal polynomials for various angles $\theta $ (columns of Table~\ref{improvementtable}) and obtain
\[
\delta_n\leq \alpha_n(\theta) \sin^n(\theta/2)\text{Lev}(n,\theta)
\]
with maximal $r$ (in Proposition~\ref{liftsphere}), where $\alpha_n(\theta)$ are the entries of the table. Note that in Table~\ref{improvementtable}, the improvement factors appear to gradually become independent of $\theta$ as $n$ enlarges. We conjecture that they tend to $\frac{1}{e}$ as $n\rightarrow\infty$.

\begin{table}[h!]
\centering
\addtolength{\leftskip} {-2cm}
\addtolength{\rightskip}{-2cm}
\scalebox{0.60}{
\begin{tabular}{c|c|c|c|c|c|c|c|c|c|c|c|c|c|c|c|c|c|c|c|c|c|c|c|c|c|c|c|c|c|c}
\backslashbox{$n$}{$\theta$}& $61^{\circ}$ & $62^{\circ}$ & $63^{\circ}$ & $64^{\circ}$ & $65^{\circ}$ & $66^{\circ}$ & $67^{\circ}$ & $68^{\circ}$ & $69^{\circ}$ & $70^{\circ}$ & $71^{\circ}$ & $72^{\circ}$ & $73^{\circ}$ & $74^{\circ}$ & $75^{\circ}$ & $76^{\circ}$ & $77^{\circ}$ & $78^{\circ}$ & $79^{\circ}$ & $80^{\circ}$ & $81^{\circ}$ & $82^{\circ}$ & $83^{\circ}$ & $84^{\circ}$ & $85^{\circ}$ & $86^{\circ}$ & $87^{\circ}$ & $88^{\circ}$ & $89^{\circ}$ & $90^{\circ}$\\
\hline
4&.942&.889&.839&.793&.750&.711&.674&.640&.608&.578&.550&.524&.500&.477&.456&.436&.417&.399&.392&.396&.400&.405&.409&.413&.416&.418&.420&.420&.420&.419\\
5&.928&.863&.803&.748&.698&.653&.611&.572&.537&.504&.474&.446&.420&.400&.374&.371&.377&.382&.387&.392&.396&.401&.404&.408&.410&.412&.412&.412&.411&.408\\
6&.915&.838&.768&.706&.650&.599&.553&.512&.474&.439&.408&.385&.389&.393&.397&.368&.374&.379&.384&.389&.393&.398&.401&.404&.406&.407&.407&.406&.404&.401\\
7&.901&.813&.735&.666&.605&.550&.501&.457&.418&.395&.373&.378&.383&.388&.391&.395&.371&.377&.382&.387&.391&.395&.398&.401&.403&.404&.403&.402&.400&.396\\
\hline
8&.888&.789&.704&.629&.563&.505&.454&.409&.394&.394&.393&.373&.378&.383&.387&.391&.369&.374&.380&.385&.389&.393&.396&.399&.400&.401&.401&.399&.396&.392\\
9&.874&.766&.673&.593&.524&.464&.411&.389&.391&.392&.392&.391&.373&.378&.383&.387&.391&.372&.378&.383&.387&.391&.394&.397&.398&.399&.398&.397&.394&.389\\
10&.862&.744&.644&.560&.488&.426&.382&.386&.389&.390&.391&.391&.389&.374&.379&.384&.388&.371&.376&.381&.385&.389&.393&.395&.397&.397&.397&.395&.391&.387\\
11&.849&.722&.617&.528&.454&.391&.378&.382&.386&.388&.390&.390&.389&.370&.376&.380&.385&.369&.374&.379&.384&.388&.391&.394&.395&.396&.395&.393&.390&.385\\
\hline
12&.836&.701&.590&.498&.422&.384&.387&.379&.383&.386&.388&.389&.389&.387&.372&.377&.382&.386&.373&.378&.383&.387&.390&.393&.394&.395&.394&.392&.388&.384\\
13&.824&.681&.565&.470&.393&.380&.384&.375&.380&.383&.386&.388&.388&.387&.385&.375&.380&.384&.371&.376&.381&.385&.389&.392&.393&.394&.393&.391&.387&.382\\
14&.811&.661&.540&.444&.387&.376&.380&.384&.377&.381&.384&.387&.388&.387&.386&.372&.377&.382&.370&.375&.380&.384&.388&.391&.392&.393&.392&.390&.386&.381\\
15&.799&.642&.517&.419&.386&.386&.377&.381&.384&.378&.382&.385&.387&.387&.386&.370&.375&.380&.384&.374&.379&.383&.387&.390&.391&.392&.391&.389&.385&.380\\
\hline
16&.788&.623&.495&.395&.385&.386&.384&.378&.382&.376&.380&.383&.386&.386&.386&.384&.373&.378&.382&.373&.378&.382&.386&.389&.391&.391&.390&.388&.385&.380\\
17&.776&.605&.474&.381&.384&.385&.385&.375&.380&.383&.378&.382&.384&.386&.386&.384&.371&.376&.381&.371&.377&.381&.385&.388&.390&.391&.390&.388&.384&.379\\
18&.764&.587&.453&.378&.382&.384&.385&.383&.377&.381&.376&.380&.383&.385&.385&.384&.382&.374&.379&.370&.376&.380&.384&.387&.389&.390&.389&.387&.383&.378\\
19&.753&.570&.434&.383&.380&.383&.384&.384&.375&.379&.373&.378&.382&.384&.385&.384&.382&.373&.378&.369&.375&.379&.383&.387&.389&.389&.389&.387&.383&.378\\
\hline
20&.742&.553&.415&.381&.377&.381&.383&.384&.382&.377&.381&.376&.380&.383&.385&.384&.383&.371&.376&.381&.374&.379&.383&.386&.388&.389&.388&.386&.382&.377\\
21&.731&.537&.397&.379&.382&.379&.382&.383&.383&.375&.379&.374&.379&.382&.384&.384&.383&.369&.375&.379&.373&.378&.382&.385&.388&.388&.388&.386&.382&.377\\
22&.720&.522&.383&.376&.380&.377&.381&.383&.383&.381&.377&.381&.377&.381&.383&.384&.383&.380&.373&.378&.372&.377&.381&.385&.387&.388&.388&.385&.382&.376\\
23&.709&.506&.383&.382&.375&.381&.379&.382&.383&.382&.375&.379&.376&.380&.382&.384&.383&.381&.372&.377&.371&.376&.381&.384&.387&.388&.387&.385&.381&.376\\
\hline
24&.699&.492&.382&.382&.376&.380&.377&.381&.382&.382&.373&.378&.374&.378&.381&.383&.383&.381&.371&.376&.370&.375&.380&.384&.386&.387&.387&.385&.381&.375\\
25&.688&.477&.381&.382&.381&.378&.375&.379&.382&.382&.380&.376&.372&.377&.381&.383&.383&.381&.370&.375&.369&.374&.379&.383&.386&.387&.387&.384&.380&.375\\
26&.678&.463&.379&.381&.381&.376&.380&.378&.381&.382&.381&.375&.379&.376&.380&.382&.383&.382&.379&.374&.369&.374&.379&.383&.385&.387&.386&.384&.380&.375\\
27&.668&.450&.377&.380&.381&.380&.378&.376&.380&.381&.381&.373&.377&.375&.379&.381&.382&.382&.379&.373&.378&.373&.378&.382&.385&.386&.386&.384&.380&.375\\
\hline
28&.658&.437&.380&.379&.381&.381&.376&.375&.378&.381&.381&.379&.376&.373&.378&.381&.382&.382&.379&.372&.377&.373&.372&.382&.384&.386&.386&.384&.380&.375\\
29&.648&.424&.379&.378&.380&.381&.375&.379&.377&.380&.381&.380&.375&.372&.377&.380&.382&.382&.380&.371&.376&.372&.377&.381&.384&.386&.385&.384&.380&.374\\
30&.639&.411&.377&.376&.379&.381&.379&.377&.376&.379&.381&.380&.373&.378&.376&.379&.381&.382&.380&.370&.375&.372&.377&.381&.384&.385&.385&.383&.379&.374\\
31&.629&.400&.379&.379&.378&.380&.380&.376&.374&.378&.380&.380&.372&.377&.374&.378&.381&.382&.380&.369&.374&.371&.376&.380&.383&.385&.385&.383&.379&.374\\
\hline
32&.620&.388&.380&.377&.377&.380&.380&.374&.378&.377&.380&.380&.378&.375&.373&.378&.380&.381&.380&.377&.374&.370&.376&.380&.383&.385&.385&.383&.379&.374\\
33&.611&.379&.380&.376&.375&.379&.380&.378&.377&.376&.379&.380&.379&.374&.372&.377&.380&.381&.380&.371&.373&.370&.375&.379&.383&.384&.385&.383&.379&.374\\
34&.602&.378&.380&.379&.378&.378&.380&.379&.376&.375&.378&.380&.379&.373&.371&.376&.379&.381&.380&.378&.372&.369&.375&.379&.382&.384&.384&.383&.379&.374\\
35&.593&.377&.379&.379&.377&.377&.379&.379&.375&.374&.378&.380&.379&.372&.376&.375&.379&.381&.380&.378&.371&.369&.374&.379&.382&.384&.384&.382&.379&.373\\
\hline
36&.584&.375&.379&.379&.376&.375&.378&.379&.373&.377&.377&.379&.379&.377&.376&.374&.378&.380&.380&.378&.371&.376&.374&.378&.382&.384&.384&.382&.379&.373\\
37&.575&.378&.378&.379&.374&.374&.378&.379&.378&.376&.376&.379&.379&.378&.375&.373&.377&.380&.380&.379&.370&.375&.373&.378&.381&.384&.384&.382&.379&.373\\
38&.567&.376&.376&.379&.378&.377&.377&.379&.378&.375&.375&.378&.379&.378&.374&.373&.377&.380&.380&.379&.369&.375&.373&.378&.381&.383&.384&.382&.378&.373\\
39&.559&.375&.375&.378&.379&.376&.376&.378&.378&.374&.374&.377&.379&.378&.373&.372&.376&.379&.380&.379&.376&.374&.372&.377&.381&.383&.384&.382&.378&.373\\
\hline
40&.550&.378&.377&.377&.379&.375&.375&.378&.379&.373&.373&.377&.379&.378&.372&.376&.376&.379&.380&.379&.376&.374&.372&.377&.381&.383&.383&.382&.378&.373\\
41&.542&.379&.376&.376&.378&.377&.377&.377&.378&.377&.376&.376&.378&.379&.376&.376&.375&.378&.380&.379&.376&.373&.372&.376&.380&.383&.383&.382&.378&.373\\
42&.534&.378&.375&.375&.378&.378&.376&.376&.378&.377&.375&.375&.378&.379&.377&.375&.374&.378&.380&.379&.376&.372&.371&.376&.380&.382&.383&.382&.378&.373\\
43&.526&.378&.378&.374&.377&.378&.375&.375&.378&.378&.374&.374&.378&.379&.377&.374&.374&.377&.379&.379&.377&.372&.371&.376&.380&.382&.383&.381&.378&.373\\
\hline
44&.518&.377&.378&.376&.377&.378&.374&.374&.377&.378&.373&.373&.377&.378&.377&.373&.373&.377&.379&.379&.377&.371&.370&.375&.379&.382&.383&.381&.378&.372\\
45&.510&.377&.378&.375&.376&.378&.377&.373&.377&.378&.372&.372&.376&.378&.378&.372&.372&.376&.379&.379&.377&.371&.370&.375&.379&.382&.383&.381&.378&.372\\
46&.503&.376&.378&.374&.375&.378&.377&.376&.376&.378&.376&.376&.376&.378&.378&.372&.371&.376&.379&.379&.377&.370&.370&.375&.379&.382&.382&.381&.378&.372\\
47&.495&.374&.377&.377&.374&.377&.377&.375&.376&.378&.377&.375&.375&.378&.378&.371&.371&.375&.378&.379&.377&.370&.369&.375&.379&.381&.382&.381&.378&.372\\
\hline
48&.488&.376&.377&.377&.376&.376&.377&.374&.375&.377&.377&.374&.374&.377&.378&.376&.375&.375&.378&.379&.378&.369&.369&.374&.378&.381&.382&.381&.378&.372\\
49&.481&.375&.376&.377&.375&.376&.377&.373&.374&.377&.377&.373&.374&.377&.378&.376&.374&.374&.378&.379&.378&.369&.374&.374&.378&.381&.382&.381&.377&.372\\
50&.474&.374&.375&.377&.374&.375&.377&.376&.373&.377&.377&.372&.373&.377&.378&.376&.374&.374&.377&.379&.378&.374&.374&.374&.378&.381&.382&.381&.377&.372\\
\hline
60&.408&.375&.376&.376&.374&.376&.374&.375&.376&.373&.374&.376&.376&.374&.375&.377&.376&.373&.374&.377&.378&.376&.370&.371&.376&.379&.381&.380&.377&.371\\
70&.376&.374&.376&.375&.375&.373&.375&.375&.372&.375&.375&.372&.375&.375&.373&.374&.376&.375&.373&.374&.377&.377&.374&.369&.374&.378&.380&.379&.376&.371\\
80&.375&.374&.374&.374&.374&.374&.374&.373&.375&.373&.375&.374&.372&.375&.375&.373&.375&.376&.373&.371&.375&.377&.375&.370&.372&.377&.379&.379&.376&.371\\
90&.373&.374&.374&.374&.373&.374&.373&.374&.372&.374&.373&.374&.374&.372&.375&.374&.372&.375&.375&.371&.373&.376&.376&.372&.371&.375&.378&.379&.376&.370\\
\hline
100&.373&.374&.373&.373&.374&.372&.374&.372&.374&.372&.374&.373&.374&.373&.373&.374&.371&.373&.375&.373&.371&.375&.376&.373&.369&.374&.378&.378&.376&.370\\
110&.372&.372&.374&.373&.372&.374&.372&.374&.372&.373&.372&.374&.372&.374&.372&.374&.374&.370&.374&.374&.371&.373&.375&.374&.371&.373&.377&.378&.375&.370\\
120&.373&.373&.373&.373&.373&.372&.373&.372&.373&.373&.373&.373&.373&.373&.373&.371&.374&.372&.373&.374&.372&.372&.375&.374&.370&.373&.376&.377&.375&.370\\
130&.373&.373&.372&.372&.373&.373&.372&.373&.373&.371&.373&.373&.373&.372&.373&.370&.373&.373&.371&.374&.373&.370&.374&.375&.371&.372&.376&.377&.375&.370
\end{tabular}}
\caption{ Table of improvement factors $\alpha_n(\theta)$
\label{improvementtable}}
\end{table}

\bibliographystyle{alpha}

\bibliography{final}
\end{document}